\documentclass[12pt]{amsart}
\usepackage[english]{babel}
\usepackage[utf8]{inputenc} 
\usepackage[T1]{fontenc}
 
\usepackage[margin=1.5in]{geometry}

\usepackage[colorinlistoftodos]{todonotes}
\usepackage[colorlinks=true, allcolors=blue]{hyperref}

\usepackage[alphabetic]{amsrefs}

\usepackage[makeroom]{cancel}
\usepackage{setspace}
\usepackage{amsfonts} 
\usepackage{amssymb}
\usepackage{amsthm} 
\usepackage{indentfirst}
\usepackage{enumerate}
 \usepackage{enumitem}   
\usepackage{amsmath}
\usepackage{graphicx}

\usepackage{mathabx}
\graphicspath{ {./images/} }

\usepackage{enumitem}
\usepackage[most]{tcolorbox}

\usepackage{tikz-cd}

\definecolor{myGreen}{RGB}{10 100 10}

\newcommand{\bC}{\mathbb{C}}

\newcommand{\bR}{\mathbb{R}}

\newcommand{\bF}{\mathbb{F}}
\newcommand{\bN}{\mathbb{N}}

\newcommand{\bA}{\mathbb{A}}

\newcommand{\cM}{\mathcal{M}}

\newcommand*{\tran}{\mathsf{T}}

\newtheorem{theorem}{Theorem}[section]

\newtheorem{prop}[theorem]{Proposition}

\newtheorem{cor}[theorem]{Corollary}
\newtheorem{thm}[theorem]{Theorem}
\newtheorem{lem}[theorem]{Lemma}

\newtheorem*{cor*}{Corollary}
\newtheorem*{thm*}{Theorem}
\newtheorem*{lem*}{Lemma}

\newtheorem*{prop*}{Proposition}

\theoremstyle{definition}

\newtheorem*{defn*}{Definition}

\theoremstyle{remark}

\title[Regularity of compact convex sets--classical and nc]{\textbf{Regularity of compact convex sets--classical and noncommutative}}

\author{David P. Blecher}
	
\address{Department of Mathematics, University of Houston, Houston, TX 77204-3008.}
	
\email{dpbleche@central.uh.edu}

\date{Revision of 9/5/2026.   To appear J. Complex Analysis and Operator Theory}

\begin{document}
    \begin{abstract}     The classical theory of regularity of embeddings of compact convex sets
    was developed in the 1970s, exclusively in the real case, and even there it does not appear to have been stated in its simplest form. 
    We begin by revisiting this setting, showing that under a reasonable condition, every locally convex topological vector space
    $E$  that contains and is spanned by a compact convex set lying in a hyperplane not passing through the origin, is a (specific) 
dual Banach space equipped  with the weak* topology.      Second, we establish the corresponding regularity theory for convex sets in complex LCTVS's. Third, we develop a theory of regular embeddings for complex noncommutative convex sets, in the sense of Davidson and Kennedy. Finally, we use the complex theory to derive a theory of regular embeddings for real noncommutative convex sets. Interestingly, at present there appears to be no direct route to the latter.   \end{abstract} \maketitle

\section{Introduction} 
The classical theory of regularity of embeddings of a compact convex set $K$ 
    was developed over sixty years ago. Regular embeddings of  $K$ in an LCTVS $E$  are particularly well-behaved affine embeddings.  Every compact convex 
    subset of an LCTVS has a regular embedding in some possibly different LCTVS.
    One of their principal features is that they allow
    $E$ to be  replaced by a topologically equivalent LCTVS which is nicer  from a functional analytic perspective, thereby abstracting  the convex set from its particular ambient space while still retaining its essential convexity properties. For many purposes, this is clearly a desirable tool.  
    
    This theory was developed only in the real case, and even there does not appear in the literature in its simplest form, as far as we have been able to determine. We first revisit this setting,    showing that under a `reasonable condition',  every LCTVS $E$
containing and spanned by a compact convex set $K$ which lies in a hyperplane not passing through the origin, `is' a 
(specific) dual Banach space with the weak* topology.    In this case (that is, if this reasonable condition holds
and $E = {\rm Span}(K)$) then we say that $K$ is embedded {\em regularly} in $E$.  
  The `containing and spanned by' here are mild conditions for some purposes.
Indeed one may often ensure the  spanning by $K$ by a  `cut down': replacing $E$ by Span$(K)$;
and we can ensure the hyperplane condition by translating $K$ and by adding one dimension to $E$  if necessary.   
The aforementioned `reasonable  condition'\footnote{Curiously, this is not Alfsen's definition of regularity, rather his definition 
has the appearance of being far more restrictive.  Alfsen gives no references to the literature here except for the 
phrase `folklore'.}  is that every $f \in A(K)$ has a continuous linear extension to $E$, that is,  $A(K;E) = A(K)$ in Alfsen's notation \cite{Alfsen}.  
Here $A(K)$ is the function system of continuous affine functions on $K$. 
 This is a condition that one would wish to hold anyway in very many applications in convexity theory.  

    Secondly, we find the matching `regularity theory' for convex sets in complex LCTVS's.  Third,  we develop the 
    theory of  regular embeddings of complex nc convex sets (that is,  nc convex sets in the sense of Davidson and Kennedy).
   Matt Kennedy informed us when we recently brought this up in conversation, 
    that the full theory of regular embeddings of nc convex sets (generalizing the classical case in e.g.\
    \cite[Section II.2]{Alfsen}) had not yet been worked out at that time.   (There is a partial form of regularity which may be found used in a couple of places in \cite{DK},  a nc hyperplane condition, which Davidson and Kennedy added 
    to overcome a point raised by Magajna.  This was also used in sequel papers by Kennedy and his collaborators and students.)    
 Then we use this to develop a  theory of  regular embeddings of real nc convex sets.   Interestingly, at this point there 
    seems no direct route to the latter, we had to obtain the real theory of  nc regular embeddings  
    via the complex theory.  We show that any complex nc regular embedding is classically regular.

The `regularity fact' at the start of our paper may be restated as: if an LCTVS $E$
contains and is  spanned by a compact convex set $K$, and if the `reasonable condition' about the existence of a linear extension 
in the first paragraph holds,
then $E \cong A(K)^*$ as LCTVS's (with the embedding of $K$ taken to its canonical embedding in $A(K)^*$. 
We will refer to this fact as   `regularity-by-extension'.  
Thus in the nc case it seems reasonable to assume (as is almost always
done in \cite{DK,BMcI,BMcII} etc), that our LCTVS is a dual operator space $F^*$ with the weak* topology, and that $K$ is  a nc compact convex set in $F^*$.  

(We remark that in a sequel paper \cite{BP2} we show that  interesting results which are slightly weaker
than our main results here, follow assuming only what we call `preregularity' below as opposed to regularity.   See the Acknowledgement at the end of our paper for a few more details.) 

In most of our paper we focus on nc convexity as opposed to 
its relative, matrix convexity.   However we take the time to state
 a simple regularity result for the latter setting. 
 (We also mention that there are some regularity type results in this setting in the sequel paper \cite{BP2}.)  For a LCTVS $E$ and finite $n \in \bN$ we give $M_n(E)$  the product topology.  
 We write $\bA(K)$ for Webster and Winkler's continuous matrix affine maps \cite{WW}.   We  write  $\delta_x$ for the evaluation map $f \mapsto f(x)$ for $x \in K$.  Then 
we have:

\begin{thm} \label{ww} Suppose that $K$ is a compact matrix convex set in the sense of 
{\rm \cite{WW}},  in a complex LCTVS $E$.  The embedding 
$\delta : K_1 \to \bA(K)^*$ extends to a topological isomorphism  $E \cong \bA(K)^*$ which is 
a homeomorphism at all finite matrix levels if and only if $E = {\rm Span}(K_1)$ and every $f \in \bA(K)$ has a continuous linear extension to $E$.
In the last condition we may also replace $\bA(K)$ with $A(K_1)$; hence the canonical map 
$\delta : K_1 \to \bA(K)^*$ extends to a topological isomorphism  $E \cong \bA(K)^*$ if and only if 
$K_1$ is regularly embedded in $E$.
\end{thm}

\begin{proof}  Note  that  $A(K_1)$ and $\bA(K)$ are bicontinuously isomorphic
and unitally order 
isomorphic, via the canonical restriction map,  by  \cite[Theorem 2.5.8]{DK}.  
So $A(K_1)^*$ and $\bA(K)^*$ are bicontinuously isomorphic.
 The one direction then follows since $A(K_1)^*$ is spanned by $K_1$, and every element of a 
Banach space $X$ has a continuous linear extension to $X^*$.  
For the other direction, since $K_1$ is regularly embedded in $E$ it follows by  `regularity-by-extension' that 
$\delta : K_1 \to A(K_1)^*$ extends to a topological isomorphism  $E \cong A(K_1)^* \cong \bA(K)^*$.
Since $M_n(E)$ has the product topology, as does $M_n(\bA(K)^*)$ (see e.g.\ \cite[Section 1.4]{BLM}), these spaces are linearly topologically isomorphic.
\end{proof} 

{\bf Remark.} 
 It is likely that there is a regularity theorem generalizing the `regularity-by-extension' fact above to the context of Webster's local operator spaces
 \cite{EWe,Wth},
 which are a `matricial generalization' of LCTVS's.   Namely, it should be the case that 
  if a local operator space $E$ contains and is  nc-spanned by a nc compact convex set $K$, and if a `reasonable condition' holds
  matching `regularity-by-extension',
then $E \cong \bA(K)^*$ `matrix-homeomorphically' (at all, including infinite, levels). However since there seems little interest at the present time in 
 such generalizations of LCTVS's (beyond the several papers of Dosi, etc) we will not pursue this here.   
 
 \medskip
 
 We state some notation and background facts.  We write $\bF$ for either $\bR$ or $\bC$, and $H$ for a Hilbert space. 
 We assume that the reader is familiar with basic convexity theory.  We denote the convex hull of a set $A$ by ${\rm co} (A)$ and the closed convex hull by $\overline{{\rm co}} (A)$.  
 We will be considering convex sets $K$ in real and complex topological vector spaces, which may also be normed spaces, operator spaces, or dual operator spaces. 
An ordered vector space is a vector space with a proper positive cone. For an ordered vector space $E$, the cone of positive elements will be denoted by $E_+$. 
   By a {\em $*$-vector space} we mean a vector space with an involution 
(a period 2 automorphism).  If $\bF = \bC$ we assume that the involution is conjugate linear.   If $A$ is a set we denote the set of selfadjoint (resp.\ positive) elements in $A$ by $A_{\rm sa}$ (resp.\ $A_+$).  We assume that the positive cone of an ordered $*$-vector space is contained in the selfadjoint part of the space.

An {\em order unit} for an ordered vector space is an element $e$ such that for all selfadjoint $x$, there exist real $r > 0$ such that $re \ge x$. The order unit is {\em archimedean} if $re + x \ge 0$ for all $r >0$ implies $x$ is positive. A real {\em archimedean order unit space} (aou space) is a real ordered vector space with an archimedean order unit.  It admits a norm defined by 
$\| a \| = \inf \{ t > 0 : -t e \leq a \leq t e \}$, the {\em order unit norm}.   Kadison's theorem(s) say that the real aou  spaces `are exactly' 
the  function spaces/systems, and those which are complete in the order unit norm `are exactly' the function systems $A(K)$ for compact convex sets $K$ (see e.g.\ \cite{Alfsen,AS1}, Section 4.3 of  \cite{KRI}). 
We discuss {\em complex archimedean order spaces} later.  

An \emph{operator space} is a subspace of $B(H)$, the bounded operators on a Hilbert space $H$, or abstractly it is a 
 vector space $E$ with a norm $\| \cdot \|_n$ on the space of matrices $M_n(E)$ for each $n \in \bN$ satisfying the axioms of Ruan's characterization (see e.g. \cite{ER}).  If $X \subseteq B(H)$ is an operator space, then the `matrix norms' above are given by identifying 
 $M_n(B(H))$ with the bounded operators on $n$-fold direct sum of copies of $H$. If $T : X \to Y$ we write $T^{(n)}$ for the canonical `entrywise' amplification taking $M_n(X)$ to $M_n(Y)$, i.e. $T^{(n)}([x_{ij}]) = [T(x_{ij})]$.  
The completely bounded norm is $\| T \|_{\rm cb} = \sup_n \, \| T^{(n)} \|$, and $T$ is completely  contractive if  $\| T \|_{\rm cb}  \leq 1$. 
For a real or complex  operator space $E$ and a possibly infinite cardinal, $M_n(E)$ is the space of matrices whose `finitely supported' submatrices 
 have uniformly bounded norm.     In the case $E = \bF$ write $M_n = M_n(\bF)$. 
Thus  $M_n$ is $B(\ell^2_n)$.  Indeed for every Hilbert space $H$, $B(H) \cong M_n$ $*$-isomorphically for some $n$, 
 after one chooses an orthonormal basis.

An {\em operator system} is a unital selfadjoint subspace of $B(H)$. We denote the identity operator in $M_n(B(H))$ by $I$ or $I_n$.  There is also an abstract characterization due to Choi and Effros in the complex case \cite{CE}. 
A map $T$ is said to be {\em positive} if it takes positive elements to positive elements, and  {\em completely positive} if $T^{(n)}$ is  positive for all $n \in \bN$. A {\em ucp map} is  unital, linear, and completely positive. A \emph{state} on an operator system or aou space $V$ is a (selfadjoint) positive unital functional into $\bF$,  or equivalently a contractive unital functional, and  $S(V)$ or $S_{\bF}(V)$  is the state space, the (compact convex) set of states.   Of course $T$ is {\em selfadjoint} if $T(x^*) = T(x)^*$ for $x \in X$.  This is automatic for completely positive maps between real or complex operator systems.  The isomorphisms (resp.\ embeddings) of operator systems which are used in this paper are bijective (resp.\ injective) ucp maps whose inverse (resp.\ inverse in its range) is ucp.  These are called unital complete order isomorphisms (resp.\ unital complete order embeddings); or unital {\rm coi} for short. 

For general background on complex operator systems and spaces, we refer the reader to e.g. \cite{Pnbook,P,BLM,DK, Dav} and in the real case to e.g.\ \cite{BReal, BR,BMcI}. 
The theory of complex   $C^*$- and von Neumann algebra theory may be found in e.g.\ \cite{Ped}, and 
 basic real $C^*$- and von Neumann algebra theory in \cite{Li}.  
The connection between complex operator systems and compact nc convex sets may be found in \cite{WW,DK,Dav}. 
The real case  may be found in  \cite{BReal, BR,BMcI,BH}.   We write  ncS$(V)$ for the {\em noncommutative state space} of $V$, namely the nc compact convex set $({\rm UCP}(V,M_n))$.

A  \emph{hyperplane} 
not passing through 0 in a vector space $E$ will be a set of the form $\{x \in E: f(x) = 1\}$ for a linear functional $f$ on $E$. A {\em nc hyperplane} in $E$ not passing through 0, is the sequence $(H_n)$ of sets $H_n = \{x \in M_n(E): f^{(n)}(x) = I_n\}$, where $f$ is a fixed linear functional on $E$, and $I_n$ is the $n \times n$ identity matrix.

 See e.g.\ \cite[Sections 1.4 and 1.6]{BLM} for the duality of operator spaces, which is crucial to nc convexity.   If $E = F^*$ is a dual operator space then the weak* continuous completely bounded  linear maps 
 $\Psi : E \to M_n$ correspond to the elements $[\psi_{ij}]$ of $M_n(F)$, via the complete isometry 
 $\Psi(x) =  [\langle x , \psi_{ij} \rangle]$. 
 That is, it is well known in operator space theory that w*CB$(E,M_n) \cong M_n(E_*)$.  This follows for example from the 
 relation  $CB((M_n)_*,F) = M_n(F)$ and $CB(X,F) \cong w^*{\rm CB}(F,X^*)$ via the adjoint.  See e.g.\ (3.2) in \cite{ER}.   
 
 We write $A(K)$  or $A(K,\bF)$ or $A_{\bF}(K)$ for the affine continuous scalar functions on a compact convex set $K$.  These are unital selfadjoint subspaces of $C(K, \bF)$, the continuous functions on $K$ with values in the field $\bF = \bR$ or $\bC$. 
 Alfsen writes  $A(K;E)$ for the set of restrictions to $K$ of continuous linear functionals on $E$, and he shows that $A(K;E)$ is always dense in $A(K)$.

For a (real or complex)  operator space $E$ we  define $\cM(E) = \bigsqcup_n \, M_n(E)$ (for $n$ cardinals bounded by some cardinal $\kappa$) , with $M_n(E)$ the matrix space of $E$.  Here $\bigsqcup_n$ is the disjoint union.
        For $X \subseteq \cM(E)$ define $X_n = X \bigcap M_n(E)$. In the case $E = \bR$ write $\cM = \cM(\bR)$.
          Then $X = \bigsqcup X_n = (X_n)$ is a {\em nc set}.
  We call $X_n$ the $n$th level of $X$.  
An embedding  $K \subseteq L$ is {\em  graded} if  $K_n \subseteq L_n$ for all $n$.
        Let  $K$ be a nc set  in a dual operator space $E$.
                We say that $K$ is  {\em nc convex} (over $E$)
                if $K$ is 
   \begin{enumerate}
            \item  graded: in particular $K_n \subseteq M_n(E)$ for all $n$, 
            \item closed under direct sums: $\sum \alpha_i x_i \alpha_i^* \in K_n$ for all bounded families $\{x_i \in K_{n_i}\}$ and every family of isometries $\{\alpha_i \in M_{n,n_i}\}$ where $\sum \alpha_i \alpha_i^* = 1_n$, 
            \item closed under compressions: $\beta^* x \beta \in K_m$ for every $x \in K_n$ and isometry $\beta \in M_{n,m}$.
        \end{enumerate} 
        As in \cite{DK} we say that $K$ is {\em closed/compact} if  $K_n$ is closed/compact in the 
        weak* topology in $M_n(E)$
         (which we recall is a dual space). 
         For $\{x_i \in M_{n_i}(E)\}$ bounded and $\{\alpha_i \in M_{n_i,n} \}$ such that $\sum \alpha_i^* \alpha_i = 1_n$, a {\em nc convex combination} of $x_i$ is defined as $\sum \alpha_i^* x_i \alpha_i \in M_n(E)$. As in  Proposition 2.2.8 in \cite{DK}
        or Remark 16.4.3 (3) in \cite{Dav},  a subset $K \subseteq \cM(E)$ is nc convex if and only if it is closed under nc convex combinations.
        
           The natural morphisms between nc convex sets are {\em nc affine functions}. These will be maps $\theta: K \to L$ between real nc convex sets which are graded, respect direct sums, and equivariant with respect to isometries. That is, for all $n$
        \begin{enumerate}
            \item $\theta(K_n) \subseteq L_n$,
            \item $\theta(\sum \alpha_i x_i \alpha_i^*) = \sum \alpha_i \theta(x_i) \alpha_i^*$ for all bounded families $\{x_i \in K_{n_i}\}$ and every family of isometries $\{\alpha_i \in M_{n,n_i}\}$ where $\sum \alpha_i \alpha_i^* = 1_n$,
            \item $\theta(\beta^* x \beta) = \beta^* \theta(x) \beta$ for every $x \in K_n$ and isometry $\beta \in M_{n,m}$.
        \end{enumerate}
    We say that     $\theta$ is continuous if $\theta|_{K_n}: K_n \to M_n(\bR)$ is continuous for every $n$. 
    Then $\bA(K)$  or $\bA_{\bF}(K)$ is the space of all continuous affine nc `scalar functions' from $K$ into $\cM(\bF)$ 
 (see \cite{DK,Dav,BMcI}). 
 We  write  ${\rm UCP}^\sigma(V, M_n)$ for the collection of weak* continuous ucp maps into $M_n$, that is, the normal matrix state space of $V$.
We  write  $$c(x,y) = \begin{bmatrix}
            x & -y\\ y & x
        \end{bmatrix} .$$ 
  We also sometimes write $c(x+iy)$ for $c(x,y)$.     
  We  write  $\delta_x$ for the evaluation map $f \mapsto f(x)$ for $x \in K$.  This is nc affine. 
             

  The complex $*$-vector space version of Kadison's 
        characterization of Archimedean order unit spaces as function systems follows immediately from the real case 
        mentioned above \cite{PT}.  See also 
e.g.\ Section 4.3 of  \cite{KRI} and \cite{WW}.  
We include a proof since it is so short:

        \begin{lem} \label{coK}   Archimedean order unit complex  $*$-vector spaces are exactly 
the complex function systems.   That is, 
let $V$ be a  complex $*$-vector space with cone $V_+ \subset V_{\rm sa}$ such that $(V_{\rm sa}, V_{+},e)$ is a real Archimedean order unit space.  Then $V$ is (complex $*$-linear unital order  embedded as) a selfadjoint unital subspace of 
       $C(K,\bC)$ for some compact Hausdorff set $K$.   \end{lem}

\begin{proof}    Let $K = S_{\bR}(V_{\rm sa})$, the real state space. 
            By Kadison's function representation
            there exists a real unital order embedding $\theta : V_{\rm sa} \to C(K,\bR)$.  It is obvious 
       that $\rho(x+iy) = \theta(x) + i \theta(y)$ defines a complex linear unital order  embedding $V \to C(K,\bC)$. 
       E.g.\ $\rho(x+iy) \geq 0$ if and only if $\theta(y) = y = 0$ and $\theta(x)$ and so $x$ are positive, that is, 
        if and only if $x+iy = x \in V_+$.   \end{proof}

        Going further in the last result and its proof, a moments further thought reveals that the restriction map $S_{\bC}(V) \to K = S_{\bR}(V_{\rm sa})$ is  surjective, one-to-one, and affine.  Indeed complex states on the complex function system 
        $V = V_{\rm sa} + i V_{\rm sa}$ are selfadjoint, and they are 
        exactly the complexifications of the real states on $V_{\rm sa}$.  So $S_{\bC}(V)$ is topologically affine isomorphic to $S_{\bR}(V_{\rm sa})$.
        In particular it is easy to see from this that the seminorm 
        $\sup_{\varphi \in S(V)} \, | \varphi(\cdot)|$ on $V$ induced by its state space, is a norm with respect to which the map  $\rho : V \to C(K,\bC)$  above is an isometry.   Also note that $\rho$ (resp.\ $\theta$) maps into (resp.\ onto) the function system $A_{\bC}(K)$ (resp.\ $A_{\bR}(K)$   
        if $V_{\rm sa}$ is complete in the order unit norm, by the classical real Kadison function representation).
        Thus $\rho$ is surjective, hence gives an isometric identification $V \cong  A_{\bC}(K)$.   See also Section 4.3 of  \cite{KRI}. 
        
        Putting the above together gives the `Kadison theorem(s)'  for complex aou spaces. See also results 1.1 and 1.2 in \cite{BH}.
        In particular we have a unital isometric order isomorphism $V \cong A(S(V))$.   There is a matching 
        duality relation, the topological affine isomorphism  $K \cong S(A(K))$ for 
        any compact convex set $K$ in 
        an LCTVS.  This is easier to prove since the 
        canonical map $\delta : K \to S(A(K)) \subset A(K)^*$ is clearly continuous, affine, and one-to-one since $E_*$ separates points of $E$.   So its range is compact.   It is surjective by a standard geometric Hahn-Banach argument.

Base norm spaces are the dual  theory to order unit  spaces.  The theory of classical and noncommutative real and complex base norm spaces is presented in the companion paper \cite{BH}.  We will use some results from the present paper there.  There are strong links  between the present paper and  \cite{BH}:  For example  if $E$ is a complex nc dual base norm space with base $K$ in the sense of  \cite{BH}, then $K$ is regularly embedded in $E$.   This follows from the proof of Theorem 4.12  in 
 \cite{BH} and our Corollary \ref{chco24} (or \ref{chco3}).  (See \cite[Corollary 4.13]{BH}.) 
 Conversely, if (classical or nc) $K$ is regularly embedded in $E$ then  $E$ `is' a  (classical or nc) base normed space with (classical or nc) base $K$. 
 
 Often in our paper for a convex or nc convex subset $K$ of a complex space $E$ we assume that there is a real subspace $W$ with $W \oplus iW = E$ and $K \subset W$.
 In many applications $E$ is a $*$-vector space and $W$ is simply the selfadjoint part $E_{\rm sa}$. 
 
 For any real unital operator subsystem $V \subseteq B$ 
 a continuous real functional on $V$ may be extended to a continuous functional on $B$. For example if $B = C_{\bR}(K)$ then the latter is a
 difference of two positive functionals, and each of these is a nonnegative multiple of  a state.  In the complex case we get 4 states by the obvious variant of this argument, or one can see this by complexification.  A similar fact holds for general $C^*$-algebras $B$ (see e.g.\ \cite{Ped}).

\section{Regularity} 

Here $K$ is a compact convex set in a LCTVS $E$. We say that the embedding of $K$ in $E$ is {\em preregular} if 
$E = {\rm Span} (K)$ 
and if $K$ lies on a hyperplane not passing through  $0$.   Assuming this we first show that 
$E \cong A(K)^*$ linearly, and that any $f \in A(K)$ is the restriction to $K$ of a linear functional on $E$.  Note that $A_{\bR}(K)$ and $A_{\bC}(K)$ are real 
and complex function systems, and are real and complex operator systems inside $C_{\bR}(K)$ and $C_{\bC}(K)$.

 \subsection{The real classical case} 
 
 \begin{lem} Let $K$ be a compact convex set in a real LCTVS $E$. There exists a surjective linear map $A(K)^* \to E$ taking
 $\delta_{x}$ to $x$ for each $x \in K$.  It is one-to-one if the embedding of $K$ in $E$ is  preregular.
 \end{lem}

\begin{proof}    Suppose that $\varphi = c_1 \, \delta_{x_1} - c_2  \, \delta_{x_2} = d_1 \, \delta_{y_1} - d_2  \, \delta_{y_2}$ on 
$A(K)$ for $x_i , y_i \in K$ and $c_i  , d_i \geq 0$.  Applying to $1$ we get that $c_1 \, - c_2 = d_1 - d_2$. 
So $c_1  \, \delta_{x_1} +  d_2  \, \delta_{y_2} = c_2  \, \delta_{x_2} + d_1 \, \delta_{y_1}$.  Dividing by 
$c_1 + d_2 = c_2+d_1$, and using that $\delta : K \to A(K)^*$ is affine and one-to-one 
(since $A(K)$ separates points), 
we obtain $c_1 \, x_1 - c_2  \, x_2 = d_1 y_1 - d_2  \, y_2$ in $E$.   This defines a linear map $A(K)^* \to E$ taking
 $\delta_{x}$ to $x$ for each $x \in K$.  Indeed as we said above any element in $A(K)^*$ is a  difference of two 
  nonnegative multiples of  a state.  Since $S(A(K)) \cong K$, every state on $A(K)$ is $\delta_{x}$ for some $x \in K$.

The argument is reversible if the embedding of $K$ is preregular.  
Indeed conversely, if  $c_1 \, x_1 - c_2  \, x_2 = d_1 y_1 - d_2  \, y_2$ in $E$ then 
using the hyperplane assumption we see that $c_1 + d_2 = c_2+d_1$, so that 
since $\delta$ is affine we get $c_1  \, \delta_{x_1} +  d_2  \, \delta_{y_2} = c_2  \, \delta_{x_2} + d_1 \, \delta_{y_1}$ and 
$ c_1 \, \delta_{x_1} - c_2  \, \delta_{x_2} = d_1 \, \delta_{y_1} - d_2  \, \delta_{y_2}$ on 
$A(K)$.   

Since $E = {\rm Span} (K)$, grouping the terms with positive coefficients and the terms with negative coefficients, and using that $K$ is convex, 
any $x \in E$ may be written as 
$c_1 \, x_1 - c_2  \, x_2$  for $x_i \in K$ and $c_i \geq 0$. So the map is surjective. 
\end{proof}

Thus if the embedding of $K$ in $E$ is  preregular  we obtain a well defined linear isomorphism $q : E \to A(K)^*$ taking $c_1 \, x_1 - c_2  \, x_2$ for  $x_i \in K$ and $c_i \geq 0$, 
to $c_1 \, \delta_{x_1} - c_2  \, \delta_{x_2}$.  It is surjective by the argument using Kadison's theorem.   So $E \cong A(K)^*$ linearly. 
 Note that $q_y(f) =  f(y)$ if $y \in K, f \in 
 A(K)$, so that $x \mapsto q_x(f)$ defines a linear functional $\tilde{f}$  on $E$ extending $f$.  
 
 We say that the (preregular) embedding of $K$ is {\em regular} if $q$ is continuous with respect to the weak* topology of $A(K)^*$. 
 In this case, $\tilde{f}$  is a continuous linear extension of $f$ to $E$.   Moreover it is the unique linear extension of $f$ to $E$ since $E  = {\rm Span} ( K)$. Indeed by the proof in the following Remark,  $\tilde{f}$ is continuous on $E$ for all $f \in A(K)$ if and only if $q$ is continuous with 
 respect to the weak* topology of $A(K)^*$.  

\bigskip 

{\bf Remark.} It is easy to see that if the embedding of $K$ is preregular, then it is regular (i.e.\ it  is preregular and $q$ is continuous with respect to the weak* topology of $A(K)^*$) if and only if  
every $f \in A(K)$ has a continuous linear extension to $E$.   Indeed assuming the latter, that extension is $\tilde{f}(x) = q_x(f)$ for $x \in E$.  So if $x_t \to x$ weak* in $E$ 
then $q_{x_t}(f) \to q_{x}(f)$ 
for each $f \in A(K)$.  The latter is simply saying that  $q_{x_t} \to q_{x}$ weak*.    

 Thus if the embedding of $K$ is regular then we have $A(K;E) = A(K)$ in Alfsen's notation of his Corollary II.2.3. 
  The definition of regularity in Alfsen assumes $q$ is a homeomorphism with respect to the weak* topology of $A(K)^*$. However this is essentially 
 automatic if $q$ is continuous, as we shall see below. 
 
 \bigskip

 Suppose that $E$ and $F$ are LCTVS's in duality, with each canonically being the dual space of each other, so that the topologies on $E$ and $F$ are the 
 weak topologies.  (This is always possible to arrange, by taking for example $F$ to be the continuous dual of $E$ with the `weak* topology' (see e.g.\ V.I in 
 \cite{Conw}), 
 and giving $E$ the $\sigma(E,F)$ topology.  Then $f_t \to f$ in $F$ 
 if and only if $\langle f_t , e \rangle \to \langle f , e \rangle$ for all $e \in E$, and similarly for nets in $E$ 
 converging $\sigma(E,F)$.) 
 Then $K$ is compact in the weak* topology induced by $E = F^*$, indeed  the map $E \to F^*$ is a homeomorphism. 
 So we have a canonical linear map $\theta : F \to A(K)$ given by restriction of $\varphi \in F$ to $K$. 
 It is one-to-one since $K$ is spanning.    
 
 Note that if $K$ is preregular then   $\theta(F)$ is dense in $A(K)$, since the annihilator of $\theta(F)$ is contained in 
 ${\rm Ker} (q^{-1}) = (0)$.    Indeed if $\varphi \in \theta(F)^\perp$ and $\varphi  = c \, \delta_{x} - d  \, \delta_{y}$ as above,  and if $\psi \in F$, then $\varphi(\psi_{|K}) = 
  c \psi(x) - d \psi(y) = \psi(cx-dy) = 0$.
 So $q^{-1}(\varphi) = cx - dy = 0$.  
 
 Note that $\theta^{-1}$ is clearly continuous with respect to either the norm or weak topology on $A(K)$, and the weak topology on $F$.
For example, $\| (\varphi_t)_{|K}
 - \varphi_{|K} \|_{C(K)} \to 0$ implies 
 $\varphi_t \to \varphi$ weakly (i.e.\ $\sigma(F,E)$) in $F$.  
 So $\theta^{-1}$ is continuous.

 \bigskip
 
{\bf Remarks.}   1)\ That the canonical embedding of 
$S(V)$ in $V^*$ is preregular, for a real unital function space $V$,  is easy.
However that the canonical embedding of 
$S(V)$ in $V^*$  is a regular embedding is not obvious.  Indeed it essentially is precisely the Kadison duality formula $V \cong A(S(V))$, since this is saying that every $f \in A(S(V))$ has a 
weak* continuous linear extension to $V^*$, and such extensions `are precisely' the elements of $V$.
Recall that by Kadison's theorem the canonical map $\iota : V \to A(S(V))$
is an isometry, and this is essentially our map $\theta$ above when $E = V^*$ and $F = V$.

In this case any affine continuous $f$ on $K$ extends norm-preservingly to a functional on $V^*$.

\smallskip

2)\ If the hyperplane in the definition of `preregular' 
 above is closed then it is $\{ x \in E : \varphi(x) = t \}$ for $\varphi \in F$ so that $\theta(\varphi /t) = 
 1$.   Conversely 
if $\theta(\varphi) = 1$ then  $\theta(\varphi)(k) = \varphi(k) = 1$, so $K$ is in the closed hyperplane 
  $\{ x \in E : \varphi(x) = 1 \}$.

 \bigskip
 
 Every (real or complex)  
 normed space $Y$ has a weak topology, which we call below the {\em weak topology associated with the norm}. 
 We will say that a normed space $Y$ is a normed predual of an LCTVS $E$ if there is a topological linear isomorphism
 between $E$ and the Banach space dual of $Y$ with its weak* topology. 
  
  \begin{thm} \label{chco} For a compact convex set $K$ which spans  
  a real LCTVS $E$, the following are equivalent:
    \begin{enumerate}
        \item [{\rm (1)}] The embedding of $K$ in  $E$ is regular. 
  \item [{\rm (2)}]  Every $f \in A(K)$ has a continuous linear extension to $E$. 
  \item [{\rm (3)}]   The map $\delta : K \to A(K)^*$ extends (necessarily uniquely)  to a 
linear  isomorphism and homeomorphism $\rho : E \to A(K)^*$, the latter with the weak* topology.
(If the embedding of $K$  is preregular this is just saying that  
the map $q$ above is a weak* homeomorphism.)  
       \end{enumerate} 
  If these hold then  
  $\theta^* = q^{-1} = \rho^{-1}$, and $E$ has a normed predual which is linearly  isometric to $A(K)$. 
  Indeed the TVS dual $(F,\tau)$ of $E$ has a complete norm with respect to which   $\tau$ agrees with the weak topology associated with the norm.  
  One may choose this norm on $F$ such that $\theta$ above is isometric, and such that
  $q$ is isometric 
  with respect to the dual Banach space norm on $E = F^*$, and is a homeomorphism for the weak* topology
  on the latter dual Banach space.  Indeed this norm on $F$  is exactly 
  the canonical order unit norm, and $F$ is an order unit space (so a function space) 
   with $F_+$ the functionals on $E$ which are positive on $K$, and with `order unit'
  the unique functional that is 1 on $K$.   \end{thm} 

 \begin{proof}   
 (1) $\Leftrightarrow$ (2) \ Follows from 
the first Remark in this subsection (the hyperplane condition in the definition of preregularity follows from the previous Remark 2). 

 (1) $\Rightarrow$ (3) \  Suppose that  the embedding of $K$ in  $E$ is regular, and that we have a bounded net $\varphi_t \to \varphi$ weak* in $A(K)^*$, which we may assume is in the ball. 
 Write $\varphi_t = c_t \delta_{x_t} - d_t \delta_{y_t}$, and $\varphi = c \delta_{x} - d \delta_{y}$, for $x_t, y_t, x, y \in K$ and nonnegative scalars
 $c_t, d_t, c, d$ with $c_t + d_t \leq 1$.   (Any contractive functional on $A(K)$ is extendible to a contractive functional on $C(K)$.  By the Hahn-Jordan theorem if necessary, the latter 
  is a difference of two positive contractive functionals on $A(K)$ whose sum is contractive (see also the lines before the start of this subsection), each of which is of form $c \delta_{x}$ by 
  the identity $K \cong S(A(K))$.) To show that $q^{-1}(\varphi_t) = c_t x_t - d_t y_t \to cx -dy = q^{-1}(\varphi)$, 
 it suffices to show that every subnet of this has a subnet that converges to $cx -dy$.  For simplicity we continue to write $\varphi_t$ for the subnet.  
 Taking further subnets and using compactness we may assume that $x_{t_\mu} \to x', y_{t_\mu}  \to y', c_{t_\mu}  \to c', d_{t_\mu}  \to d'$.  Thus $c_{t_\mu}  x_{t_\mu} - d_{t_\mu}  y_{t_\mu}  \to c'  x'- d' y'$.
 Since $q$ is continuous we have
 $c \delta_{x} - d \delta_{y} = c' \delta_{x'} - d' \delta_{y'} = \varphi$.  Thus $c' x' -d' y' =cx -dy$, as desired.  That is,    $q^{-1}(\varphi_t) \to q^{-1}(\varphi)$.
 
 For any $\psi \in F$, define a functional $G(\varphi) = \langle \psi , q^{-1} (\varphi) \rangle$ on $A(K)^*$.   For a bounded net $\varphi_t \to \varphi$ weak* in $A(K)^*$ 
 we have $G(\varphi_t ) \to G(\varphi)$.   Thus by the Krein-Smulian theorem, $G$ is weak* continuous.   Thus $q^{-1}(\varphi_t ) \to q^{-1}(\varphi)$ weakly (that is, $\sigma(E,F)$), so that 
 $q^{-1}(\varphi_t ) \to q^{-1}(\varphi)$ in $E$.  Hence $q^{-1}$ is continuous.  Thus $q$ is a homeomorphism.

 (3) $\Rightarrow$ (1) \    If $\rho$ is a homeomorphism 
then $(\rho^{-1})^*$ is onto and a homeomorphism for the weak topologies.  
However $$((\rho^{-1})^*(\psi))(k) = \psi(\rho^{-1}(\delta_k))  = \psi(k) = \theta(\psi)(k), \qquad k \in K , \psi \in F,$$ so that $\theta = (\rho^{-1})^*$.  Thus  $\theta$ is a surjective homeomorphism between the TVS dual
$F$ of $E$, and $A(K)$ (both with the weak topology).
Again, the hyperplane condition follows 
 if $1 \in \theta(F)$ by the above Remark 2).  So the embedding of $K$ in $E$ is preregular. 
 If $\theta$ is  a homeomorphism then 
 $$(\theta^{-1})^*(k)(f) = \langle k, \theta^{-1}(f) \rangle = f(k) = \delta_k(f), \qquad k \in K, f \in F,$$ so that $(\theta^{-1})^* = q$. 
 Thus $q$ is  a homeomorphism.  
 So $K$ is regularly embedded since  $q$ is continuous.

 The assertions about the 
 norm on $F$ clearly follow from  the above (one just transfers the norm from $A(K)$ via $\theta$). Let $G$ be the 
  dual Banach space of $F$, and $\tau$ its weak* topology.  The map $E \to (G,\tau)$  is a homeomorphism by 
  definition of these topologies, and
  $q = \theta^*$ is isometric.   We leave the final assertions as an exercise (indeed as we said in the introduction, $A(K)$ is the generic (complete) archimedean order unit space). 
 \end{proof}  
 
 {\bf Remarks.} 1)\ One may add as a fourth equivalence in the theorem that
  {\rm (4)} \  The map $\theta$ above is  a surjective homeomorphism between the TVS dual $F$ of $E$, and $A(K)$  (both with the weak topology).   Indeed in  the proof of the theorem we showed that  (3) implies (4). 
  Conversely, if $\theta$ is  a homeomorphism then 
 $$(\theta^{-1})^*(k)(f) = \langle k, \theta^{-1}(f) \rangle = f(k) = \delta_k(f), \qquad k \in K, f \in F,$$ so that $(\theta^{-1})^* = q$. 
 Thus (3) holds.
 
 \smallskip
 
 2)\ Since $F$ is an order unit space in the last result (and the next), its dual $E$ may be viewed as a base norm space \cite{BH}.  A similar assertion holds  in the matching theorems later in the complex case or nc case.
 
\medskip

 The following s an isometric variant.  It is rather superficial but we include it for completeness. 
 
\begin{cor} \label{chco23} For a real dual Banach space $E$ and a compact convex set $K \subseteq {\rm Ball}(E)$ which spans  
  $E$, the following are equivalent:
    \begin{enumerate}
        \item [{\rm (1)}] The embedding of $K$ in  $E$ is regular 
         and the map $q$ above is a contraction. 
  \item [{\rm (2)}]  Every $f \in A(K)$ has a continuous norm preserving linear extension to $E$. 
  \item [{\rm (3)}]   The map $\delta : K \to A(K)^*$ extends (necessarily uniquely)  to a 
weak* homeomorphic linear isometry  $\rho : E \to A(K)^*$ 
 (both spaces with the weak* topology).
   \item [{\rm (4)}]   The map $\theta$ above is  an isometric isomorphism between $F = E_*$ and $A(K)$.   
    \end{enumerate} 
  In this case   $\theta^* = q^{-1} = \rho^{-1}$.   \end{cor} 

 \begin{proof}   By the previous result any of (2)--(4) imply that the embedding of $K$ in  $E$ is regular,
 and $\theta^* = q^{-1} = \rho^{-1}$.
  It is easy to see that (2) is equivalent to (4), and that (4) is equivalent to (3), since $q = \rho$ is
 weak* continuous.   These easily imply (1).  Conversely, since $K \subseteq {\rm Ball}(E)$ we have that $\theta$  is a contraction on
 $F = E_*$, and hence so is $\theta^* = q^{-1}$.  Thus (1) implies (3).  \end{proof}

\subsection{Complex classical regularity}  \label{coca} The operator system complexification $A_{\bR}(K)_c$ of $A_{\bR}(K)$ is $A_{\bC}(K)$ \cite{BMcI}, so that $A_{\bC}(K)^* \cong (A_{\bR}(K)^*)_c$ (see also \cite{BH}).
Thus every $\varphi \in A_{\bC}(K)^*$ may be written as $c_1 \, \delta_{x_1} - c_2  \, \delta_{x_2} + i(c_3 \, \delta_{x_3} - c_4  \, \delta_{x_4} )$ for $x_i \in K$ and $c_i \geq 0$.   

Assume that $K$ is a compact convex set in a complex LCTVS $E$.
We assume that $E = W \oplus i W$ for a real subspace $W$ containing $K$.
For example if $K$ is the state space of a complex function system $V$ in $C(\Omega,\bC)$ then $W$ is the set of selfadjoint functionals
on $V$. 
We assume that $K$ spans
$E$, or equivalently that $K$ real spans $W$. 
We also assume that there is a complex linear functional $\varphi$ on $E$ with $K \subset 
\{ x \in E : \varphi(x) = t \}$ for a real or complex $t \neq 0$.  
For example if $K$ is the state space of a complex function system $V$ then one may take $t = 1$ and $\varphi = \delta_1$. If all these three conditions  hold (involving spanning, the subspace $W$, and hyperplane) then we say that the embedding of $K$ in $E$ is {\em preregular}. 
The argument below also works if $K \subseteq 
\{ x \in E : {\rm Re} \, \varphi(x) = t \}$ for a fixed real $t$.

 \begin{lem} Let $K$ be a compact convex set in a complex LCTVS $E$. There exists a surjective linear map $A(K)^* \to E$ taking
 $\delta_{x}$ to $x$ for each $x \in K$.  It is one-to-one if the embedding of $K$ in $E$ is  preregular.
 \end{lem}
 
\begin{proof}   Suppose that 
$$\varphi = c_1 \, \delta_{x_1} - c_2  \, \delta_{x_2} + i(c_3 \, \delta_{x_3} - c_4  \, \delta_{x_4} ) = d_1 \, \delta_{y_1} - d_2  \, \delta_{y_2} + i(d_3 \, \delta_{y_3} - d_4  \, \delta_{y_4} )$$ on
$A_{\bC}(K)$ for $x_i , y_i \in K$ and $c_i  , d_i \geq 0$.  Applying to $1$ we get that $c_1 \, - c_2 + i(c_3 - c_4) = d_1 - d_2
+ i (d_3 - d_4)$. So  $c_1 \, - c_2 = d_1 - d_2$
and $c_3 - c_4 = d_3 - d_4$.  
So $c_1  \, \delta_{x_1} +  d_2  \, \delta_{y_2} = c_2  \, \delta_{x_2} + d_1 \, \delta_{y_1}$.  Dividing by
$c_1 + d_2 = c_2+d_1$ and using that $\delta : K \to A_{\bC}(K)^*$ is affine and one-to-one (by the point separation of 
$A_{\bR}(K)$)
we obtain $c_1 \, x_1 - c_2  \, x_2 = d_1 y_1 - d_2  \, y_2$ in $E$.   
Similarly $c_3 \, x_3 - c_4  \, x_4= d_3 y_3 - d_4  \, y_4$ in $E$.   

 This defines a linear map $A(K)^* \to E$ taking
 $\delta_{x}$ to $x$ for each $x \in K$.  Indeed as we said above any element in $A(K)^*$ is a  difference of four 
  nonnegative multiples of  a state, and every state on $A(K)$ is $\delta_{x}$ for some $x \in K$.

The argument is reversible if 
the embedding of $K$ is preregular.  
Indeed conversely, 
suppose that  $c_1 \, x_1 - c_2  \, x_2 + i (c_3 \, x_3 - c_4  \, x_4) = d_1 y_1 - d_2  \, y_2 + i (d_3 y_3 - d_4  \, y_4)$ in $E$
 for $x_i , y_i \in K$ and $c_i  , d_i \geq 0$. 
Applying $\varphi$ we see that $(c_1-c_2) + i (c_3 \,  - c_4) = d_1-d_2 + i (d_3 \,  - d_4)$.  So $c_1 \, - c_2 = d_1 - d_2$
and $c_3 - c_4 = d_3 - d_4$.   Hence, $$\frac{c_1x_1 + d_2y_2}{c_1+d_2} = \frac{c_2x_2 + d_1y_1}{c_2+d_1} \; {\rm and} \; \frac{c_3x_3 + d_4y_4}{c_3+d_4} = \frac{c_4x_4 + d_3y_3}{c_4+d_3}.$$  Since $K$ is convex, by applying $\delta:K \rightarrow S(A(K))$, which is affine, we get
$$\frac{c_1\delta_{x_1} + d_2\delta_{y_2}}{c_1+d_2} = \frac{c_2\delta_{x_2} + d_1\delta_{y_1}}{c_2+d_1} \; {\rm and} \; \frac{c_3{x_3} + d_4\delta_{y_4}}{c_3+d_4} = \frac{c_4\delta_{x_4} + d_3\delta_{y_3}}{c_4+d_3}.$$
It follows that 
$ c_1 \, \delta_{x_1} - c_2  \, \delta_{x_2} + i(c_3 \, \delta_{x_3} - c_4  \, \delta_{x_4} ) = d_1 \, \delta_{y_1} - d_2  \, \delta_{y_2} + i(d_3 \, \delta_{y_3} - d_4  \, \delta_{y_4} )$.

Grouping the real and complex parts of the coefficients, and then the terms with positive coefficients and the terms with negative coefficients, and using that $K$ is convex,
any $x \in E$ may be written as
$c_1 \, x_1 - c_2  \, x_2 + i (c_3 \, x_3 - c_4  \, x_4)$  for $x_i \in K$ and $c_i \geq 0$. So the map is surjective. 
\end{proof}

Thus if the embedding of $K$ in $E$ is  preregular  we obtain a well defined linear isomorphism $q : E \to A(K)^*$ taking
$c_1 \, x_1 - c_2  \, x_2 + i (c_3 \, x_3 - c_4  \, x_4)$ for  $x_i \in K$ and $c_i \geq 0$,
to $c_1 \, \delta_{x_1} - c_2  \, \delta_{x_2}  + i(c_3 \, \delta_{x_3} - c_4  \, \delta_{x_4} )$.
 It is surjective by the argument using Kadison's theorem.   So $E \cong A(K)^*$ linearly. 

Note that $q_y(f) =  f(y)$ if $y \in K, f \in
 A_{\bC}(K)$, so that $x \mapsto q_x(f)$ defines a linear functional $\tilde{f}$  on $E$ extending $f$.  Moreover it is the unique extension of $f$ to $E$ since $E  = {\rm Span} ( K)$.
 We say that  the (complex preregular)  embedding of $K$ in $E$ is {\em regular}
 if $q$ is continuous with 
 respect to the weak* topology of $A(K)^*$.
 In this case $\tilde{f}$  is a continuous extension of $f$ to $E$.  Indeed as in the previous section $\tilde{f}$ is continuous on $E$ for all $f \in A(K)$ if and only if $q$ is continuous with 
 respect to the weak* topology of $A(K)^*$. 
 That is: 
  
\bigskip

{\bf Remark.} As in the previous section, the assumption that $q$ is continuous with respect to the weak* topology of $A(K)^*$ is equivalent to 
every $f \in A(K)$ having a continuous extension to $E$.   

\bigskip

 We have $A_{\bC}(K;E) = A_{\bC}(K)$ in Alfsen's notation of his Corollary II.2.3. 
  The equivalent of Alfsen's definition of regularity (matching 
 the real case) would seem to need  also that $q$ is a homeomorphism with respect to the weak* topology of $A(K)^*$. However as in the real case we shall see that this is   automatic   if $q$ is continuous.
 
 Suppose that $E$ and $F$ are 
 complex LCTVS's in duality, with each canonically being the dual space of each other.  (As in the real case this is always possible to arrange, by taking for example $F$ to be the continuous dual of $E$ with the `weak* topology' (see e.g.\ V.I in 
 \cite{Conw}), 
 and giving $E$ the $\sigma(E,F)$ topology.  
   Then $K$ is compact in the weak* topology induced by $E = F^*$, indeed  the map $E \to F^*$ is a homeomorphism. 
 So we have a canonical  linear map $\theta : F \to A_{\bC}(K)$ given by restriction of $\varphi \in F$ to $K$. 
 It is one-to-one since $K$ is spanning.   
 
 Note that   if $K$ is preregular then   $\theta(F)$ is dense in $A(K)$, since the annihilator of $\theta(F)$ is contained in 
  ${\rm Ker} (q^{-1}) = (0)$.    Indeed if $\varphi \in \theta(F)^\perp$ and $\varphi  =  c_1 \, \delta_{x_1} - c_2  \, \delta_{x_2}  + i(c_3 \, \delta_{x_3} - c_4  \, \delta_{x_4} )$  as above,  and if $\psi \in F$, then $$\varphi(\psi_{|K}) = c_1 \psi(x_1) - c_2 \psi(x_2) + i( c_3 \psi(x_3) - c_4 \psi(x_4)). $$
  But this equals $\psi(c_1 x_1 -c_2  \, x_2 + i (c_3 \, x_3 - c_4  \, x_4) ) = 0.$
 So $q^{-1}(\varphi) =  c_1 x_1 -c_2  \, x_2 + i (c_3 \, x_3 - c_4  \, x_4) = 0$.  
 Note that the hyperplane condition follows again 
  if $1 \in \theta(F)$.  Also, $\theta^{-1}$ is continuous as before with respect to either the norm or weak topology on $A(K)$, and the weak topology on $F$.

  \begin{thm} \label{chco2} For a compact convex set $K$ which spans  
  a complex LCTVS $E$,
  the following are equivalent:
    \begin{enumerate}
        \item [{\rm (1)}] The embedding of $K$ in  $E$ is regular. 
  \item [{\rm (2)}]  Every $f \in A(K)$ has a continuous extension to $E$. 
  \item [{\rm (3)}]   The map $\delta : K \to A(K)^*$ extends (necessarily uniquely)  to a 
linear  isomorphism and homeomorphism $\rho : E \to A(K)^*$, the latter with the weak* topology.
(If the embedding of $K$  is preregular this is just saying that  
the map $q$ above is a weak* homeomorphism.)    
    \end{enumerate} 
     If these hold then    $\theta^* = q^{-1} = \rho^{-1}$, and $E$ has a normed predual which is linearly  isometric to $A(K)$. 
  Indeed the TVS dual $(F,\tau)$ of $E$ has a complete norm with respect to which 
  $\tau$ agrees with the weak topology associated with the norm.  
  One may choose this norm on $F$ such that $\theta$ above is isometric, and such that
  $q$ is isometric 
  with respect to the dual Banach space norm on $E = F^*$, and is a homeomorphism for the weak* topology
  on the latter dual Banach space.  Indeed this norm on $F$  is exactly 
  the canonical order unit norm, and $F$ is a complex function system (or equivalently an order unit $*$-vector space) 
   with $F_+$ the functionals on $E$ which are positive on $K$, and with `order unit'
  the unique functional that is 1 on $K$.  
    \end{thm} 

 \begin{proof}   (1) $\Rightarrow$ (2) \ Follows from 
the first Remark in this subsection.

(2)  $\Rightarrow$ (1) \ Some of this follows from 
the first Remark in this subsection.  The hyperplane condition in the definition of preregularity follows from a remark above
as in the real case. 
If the  extension condition holds then the map $\theta : F \to A(K)$ above is bijective. 
Let $G = \{ \psi \in F : \psi(K) \subset \bR \}$.  This is a closed 
 real subspace in $F$, and $F = G \oplus iG$ since $A_{\bC}(K) = A_{\bR}(K) \oplus i A_{\bR}(K)$. 
 Now $(F^*)_r  \cong  (F_r)^*$ via the map $g \mapsto {\rm Re} \, g$, which has inverse $h \mapsto \tilde{h}$ for $h \in (F_r)^*$, 
 where $\tilde{h} = h - i h(i \cdot)$.
 Note that $h(G) = 0$ iff $\tilde{h}$ is purely imaginary on $G$, and 
 $h (iG) = 0$  iff $\tilde{h}$ is purely real on $G$.   Let $W$ be the set of functionals in $F^*$ 
 which are purely real on $G$, then $K \subset W$.
 Also, $iW$ is the set of functionals in $F^*$  
 which are purely imaginary on $G$.   
 This gives a real decomposition $E = W \oplus i W$, with $K \subset W$ as desired.     
 
 That  (1) $\Rightarrow$ (3)  follows by an obvious adaption of the argument in the real case (the difference being that the `Jordan decomposition' of a functional on $A(K)$ has four terms
 as opposed to two.  
 
As in the remark after Theorem \ref{chco}, one may add if one wishes an intermediate condition, namely:
 (4) \ The map $\theta$ above is  a surjective homeomorphism between $F$ and $A(K)$ (both with the weak topology).  
 That (3) $\Rightarrow$ (4), and $\theta^* = q^{-1}$, is as in the real case.
To see that (4) $\Rightarrow$ (1): If $\theta$ is  a homeomorphism then  $q$ is  a homeomorphism, and the hyperplane condition follows 
 if $1 \in \theta(F),$  as in the real case.  If $W = (\theta^*)^{-1} (A_{\bC}(K)^*_{\rm sa})$
 we have $W \oplus i W = E$, and it is easy to see that  $K \subset W$.
 So $K$ is regularly embedded. 

   The final assertions are as in the real case.    In the complex case
  $A(K)$ is the generic archimedean order unit $*$-vector space by the complex version of Kadison's theorem
 \ref{coK}.  
  The involution on $F$ is uniquely determined since it is trivial on $F_+$ and $F$ is spanned by $F_+$.  
 \end{proof}  

 {\bf Remark.}  Again, that the canonical embedding of 
$S(V)$ in $V^*$ is preregular is easy in the complex case.
However that the canonical embedding of 
$S_{\bC}(V)$ in $V^*$  is a regular embedding is not obvious.  Indeed it essentially is precisely the complex case of Kadison's duality formula $V \cong A(S(V))$ unitally order isomorphically and bicontinuously, since this is saying that every $f \in A(S(V))$ has a 
weak* continuous extension to $V^*$, and such extensions `are precisely' the elements of $V$.
Recall that the canonical map $\iota : V \to A(S(V))$
is bicontinuous onto its closed range (by the usual equivalence of the norm and numerical radius we get $\| \iota^{-1} \| \leq 2$), and this is essentially our map $\theta$ above when $E = V^*$ and $F = V$.

 \begin{cor} \label{chco23b} For a complex dual Banach space $E$ and a compact convex set $K \subseteq {\rm Ball}(E)$ which spans  
  $E$, the following are equivalent:
    \begin{enumerate}
        \item [{\rm (1)}] The embedding of $K$ in  $E$ is regular 
         and the map $q$ above is a contraction. 
  \item [{\rm (2)}]  Every $f \in A(K)$ has a continuous norm preserving extension to $E$. 
  \item [{\rm (3)}]    The map $\delta : K \to A(K)^*$ extends (necessarily uniquely)  to a 
weak* homeomorphic linear isometry  $\rho : E \to A(K)^*$  (both spaces with the weak* topology). 
   \item [{\rm (4)}]   The map $\theta$ above is  a surjective isometric homeomorphism between $F = E_*$ and $A(K)$.  
    \end{enumerate} 
  In this case 
  $\theta^* = q^{-1} = \rho^{-1}$.   \end{cor} 

The proof is similar to that of Corollary  \ref{chco23}. 
In the nc case we will obtain `completely isometric' and `completely isomorphic' variants of the last result.  Note that any nc affine continuous contractive (w.r.t.\ the sup norm) $f : K \to M_n$ extends to a contractive matrix in $M_n(V)
= CB((M_n)_*,V) = CB^\sigma(V^*,M_n)$.   Thus such $f$ comes from a completely contractive (weak*) continuous $\tilde{f} :
E \to M_n$.

 \section{Nc regularity}

 \subsection{Nc spans} \label{ncs} Let  $K$ be a nc set embedded in a complex dual operator space $E$.
   We say that  $K$ lies  in a (weak* closed) hyperplane $H$ not passing through 0 if  there exists $\psi \in F = E_*$ and $\gamma \in \bR, \gamma \neq 0$, such that 
 $K_n \subset H_n$ for all (finite, if one wishes) $n$, where $H_n = \{ [x_{ij}] \in M_n(E) : [\langle \psi , x_{ij}] \rangle ] = \gamma \, I_n \}$.  If  $E$ has a real subspace $W$ with $W \oplus iW = E$ and $K \subset W$ as graded nc 
 sets, and if  $K$ lies  in a (weak* closed) hyperplane $H$ not passing through 0, then we say that the embedding
 of $K$ in $E$ is {\em stately}.    Note that if the  embedding
 of $K$ in $E$ is  stately then $K_1$ lies in a hyperplane $H_1$ not passing through 0, and also lies in $W$ above.   In the case that $E$ is also a matrix ordered $*$-vector space then one often considers the case 
 that $K_n \subset M_n(E)_+$ for all $n$, and 
 then $W$ is usually be taken to be $E_{\rm sa}$.
 
 For a complex operator system $V$ the embedding of $K = {\rm ncS}(V)$ in $V^*$ is stately.
 Indeed if $W = E_{\rm sa}$ then 
 $W + iW = E$ and $K \subset W$ as graded nc sets.
 And 
 $K$ lies in the nc hyperplane $H$ 
 with $H_n$ the matrices in $M_n(V^*) = CB(V,M_n)$
 corresponding to unital maps in $CB(V,M_n)$. The nc hyperplane is defined by $1_V$ considered as an nc affine map on $E = V^*$. 
 In this case too every complex cb $u : V \to M_n$ is of the form $c_1 \, x_1 c_1 - c_2  \, x_2 c_2 + i (c_3 \, x_3 c_3- c_4  \, x_4 c_4)$  for $x_i \in K_n$ and $c_i \in M_n(\bC)^+$, by  Wittstock's representation
  of a cb map as a combination of 4 cp maps, together with the unital representation
   of cp maps in \cite[Lemma 2.2]{CE}.  This may be viewed as a nc linear combination.
 We next study general  nc linear combinations and the associated `nc span'.

 We consider `nc  spans' in $M_n(E)$  as spans or `finite linear combinations' with `matrix coefficients'.   By this we mean either `square' representations 
  $\sum_{k=1}^m \, a_k^* x_k b_k$ where $x_k \in K_n$ and $a_k, b_k
 \in M_n(\bC)$, or 
 `rectangular' representations
 $\sum_{k=1}^m \, a_k^* x_k b_k$ where $x_k \in K_{n_i}$ and $a_k, b_k
 \in M_{n_i,n}(\bC)$.  Here $m < \infty$.  We write ncSpan$_n(K_n)$ for the set of such square expressions,
 a vector subspace of $M_n(E)$, and 
 ncSpan$(K)$ for the graded set 
 of all such `rectangular' representations. By the polarization identity 
 we have 
 $$a^* x b = \frac{1}{4} \, \sum_{k=0}^3 \, i^k (b + i^k a)^* x (b + i^k a).$$
Thus we may write  $\sum_{k=1}^n \, a_k^* x_k b_k$ as a sum 
$A - B + i (C-D)$, where $A,B,C,D$ are each of form 
$\sum_{k=1}^m \, c_k^* x_k c_k$.
Here in the square case  $x_k \in K_n$ and $c_k
 \in M_n(\bC)$, while in the rectangular case
$x_k \in K_{n_i}$ and $c_k
 \in M_{n_i,n}(\bC)$.   Hence ncSpan$(K)$ (resp.\ ncSpan$_n(K_n)$) is the set of such 
 $A - B + i (C-D)$, where $A,B,C,D$ are each of the latter form.  Note that ncSpan$_n(K_n)$ is a subspace of $M_n(E)$, while ncSpan$(K)$ is a graded operator subspace of $E$, the latter considered as a graded set.
 (Note that the set of  rectangular expressions
 $d^* x d$ for some $n$ and $x \in K_{n_i}$ and $d \in M_{n_i,n}(\bC)$ is  a cone.   That is, the sum of several $c_k x_k c_k$ is again a $d^* x d$.  Similarly 
 for rectangular expressions
 $a^* x b$.)

 We say that the embedding of a nc convex set $K$ in  a complex dual operator space $E$ is
 {\em nc preregular} (resp.\ $n$-preregular) if it is stately and $E = {\rm ncSpan}(K)$ (resp.\ $M_n(E) =  {\rm ncSpan}_n(K_n)$).

\begin{cor}   \label{wascoro}If $E =  {\rm ncSpan}(K)$ then $K_1$ spans $E$.  Thus if 
 the embedding of a nc convex set $K$ in  a complex dual operator space $E$ is
  nc preregular, then the embedding of $K_1$ in $E$ is preregular.
 \end{cor}  
 
   \begin{proof} Indeed writing $x \in E$ (level 1) as $A - B + i (C-D)$ as above,
 each of the $1 \times 1$ matrices $A,B,C,D$ are of form $\sum_{k=1}^m \, c_k^* x_k c_k \in K_1$ since $K$ is nc convex.
 We said above that if the  embedding
 of $K$ in $E$ is  stately then $K_1$ lies in a hyperplane $H_1$ not passing through 0, and also lies in $W$ above. 
 \end{proof}

Next we observe that such nc linear combinations are `preserved' by linear maps $T : E \to Y$
that map $K$ into $L$ (resp.\ $T_n(K_n) \subseteq L_n$), for nc convex sets $K, L$ in $E, Y$ respectively.  
Moreover we have:

        \begin{prop} \label{dum}   Let $K$ and $L$ be nc convex sets embedded in complex 
        operator spaces
        $E, Y$ respectively, with $K$ stately embedded.   A nc affine map $u : K \to L$ (resp.\ $u : K_n \to L_n$)  extends uniquely
to a linear map $\hat{u}$ from ncSpan$(K)$  to ncSpan$(L)$ (resp.\  ncSpan$_n(K_n)$  to ncSpan$_n(L_n)$)
satisfying $\hat{u}(c^* x c) = c^* u(x) c$ for all $x \in K$ and all scalar matrices $c$  of appropriate size (resp.\ $c$ in $M_n(\bC)$).    Moreover if $L$ is also stately                                                                                                                                                                                                  embedded in $Y$ then 
$\hat{u}$ is one-to-one on ncSpan$(K)$ (resp.\ ncSpan$_n(K_n)$) if $u$ is one-to-one on $K$
(resp.\ $K_n$).
        \end{prop}
        
        \begin{proof} First define $\hat{u}$ on $\sum_{k=1}^m \, c_k^* x_k c_k$ to be 
$\sum_{k=1}^m \, c_k^* u(x_k) c_k$.  To see that the latter is well defined first suppose that $\sum_{k=1}^m \, c_k^* x_k c_k = \sum_{k=1}^r \, d_k^* y_k d_k$.   Then applying the hyperplane map 
we get 
$\sum_{k=1}^m \, c_k^*  c_k = \sum_{k=1}^r \, d_k^*  d_k$.  We may assume that $\sum_{k=1}^m \, c_k^*  c_k \leq I_n$.  Let $a$ be a square root of $I_n - \sum_{k=1}^m \, c_k^*  c_k$.  Then
  for any $z$, 
$$u_n(\sum_{k=1}^m \, c_k^*  x_k c_k + a z a)
= \sum_{k=1}^m \, c_k^* u(x_k) c_k + a u(z) a.$$
Similarly for the $d_k, y_k$ expression, and it follows that
$\hat{u}$ is well defined on $\sum_{k=1}^m \, c_k^* x_k c_k$.
Similarly $\hat{u}$  extends to be  well defined and real linear on expressions 
$\sum_{k=1}^m \, c_k^* x_k c_k - \sum_{k=m+1}^r \, c_k^* x_k c_k$.
Indeed if $\sum_{k=1}^m \, c_k^* x_k c_k - \sum_{k=m+1}^r \, c_k^* x_k c_k$ equaled a similar expression with terms $d_k^* y_k d_k$,
then one may take negative term to the other side, and then apply the previous argument.  Finally we check that 
$\hat{u}$  extends to be  well defined and complex  linear on ncSpan$(K)$.  Suppose that $v \in {\rm ncSpan}(K)$  
 is written in two ways  as a sum 
$A - B + i (C-D) = A' - B' + i (C'-D') $, where $A,B,C,D, A',B',C',D'$ are each of form 
$\sum_{k=1}^m \, c_k^* x_k c_k$.  Since $E = W \oplus iW$ it follows that $A - B = A' - B'$ and $C-D = C'-D'$. 
By  the previous argument $\hat{u}$ is well defined on each of these two, hence is  well defined on $v$. 

So $\hat{u}$ is well defined on ncSpan$(K)$ (resp.\ ncSpan$_n(K_n)$), and defines 
a linear map $\hat{u}$ from ncSpan$(K)$  to ncSpan$(L)$ (resp.\  ncSpan$_n(K_n)$  to ncSpan$_n(L_n)$).  
This map takes an expression 
$$\sum_{k=1}^m \, c_k^* x_k c_k - \sum_{k=m+1}^r \, c_k^* x_k c_k + i (\sum_{k=r+1}^s \, c_k^* x_k c_k 
- \sum_{k=s+1}^t \, c_k^* x_k c_k)$$ 
to the same expression with $x_k$ replaced by $u(x_k)$.
It is clearly the unique linear map $v$ from ncSpan$(K_n)$  to ncSpan$(L_n)$ such that $v(c^* x c) = c^* u(x) c$ for all appropriate $c, x$.  The latter property is automatic if $v = w_n$ for linear $w : {\rm Span}(K_1) \to Y$ with 
$w_n = u_n$ on $K_n$.

Similarly  $\hat{u}$ is one-to-one on ncSpan$(K)$ (resp.\ ncSpan$_n(K_n)$) if $u$ is one-to-one on $K$
(resp.\ $K_n$), and $Y$ has a real subspace $Z$ with $Z + iZ = Y$ and $L \subset Z$ as graded nc sets. 
We also suppose that $L$ lies in a nc hyperplane just as $K$ does.
   For suppose that 
$$\sum_{k=1}^m \, c_k^* u(x_k) c_k - \sum_{k=m+1}^r \, c_k^* u(x_k) c_k + i (\sum_{k=r+1}^s \, c_k^* u(x_k) c_k 
- \sum_{k=s+1}^t \, c_k^* u(x_k) c_k)$$
equalled a similar expression with the $c_k$
and  $x_k$ replaced by $d_k$ and $y_k$.
By the property above involving $Z$ we may deal with
the `real and imaginary parts' separately.
For example for the `real part'  we take all negative terms to the other side of the equation to get
$$\sum_{k=1}^m \, c_k^* u(x_k) c_k + 
\sum_{k=m'+1}^{r'} \, d_k^* u(y_k) d_k =
\sum_{k=1}^{m'} \, d_k^* u(y_k) d_k  + 
\sum_{k=m+1}^{r} \, c_k^* u(x_k) c_k.$$
Then applying the hyperplane map  we get 
$$\sum_{k=1}^m \, c_k^*  c_k + 
\sum_{k=m'+1}^{r'} \, d_k^*  d_k =
\sum_{k=1}^{m'} \, d_k^*  d_k  + 
\sum_{k=m+1}^{r} \, c_k^*  c_k.$$
We may assume that this is a contraction, $b$ say.
Let $a$ be a square root of $I_n - b$.  Then for fixed $z \in K_n$ we have 
$$u_n(\sum_{k=1}^m \, c_k^*  x_k c_k + 
\sum_{k=m'+1}^{r'} \, d_k^* y_k d_k + a z a) \\
= \sum_{k=1}^m \, c_k^* u(x_k) c_k + \sum_{k=m'+1}^{r'} \, d_k^* u(y_k) d_k + a u(z) a.$$
Similarly for $u_n$ applied to 
the $\sum_{k=1}^{m'} \, d_k^* y_k d_k  + 
\sum_{k=m+1}^{r} \, c_k^* x_k c_k$ expression. If $u$ is one-to-one on $K$
(resp.\ $K_n$) it follows that
$$\sum_{k=1}^m \, c_k^*  x_k c_k + 
\sum_{k=m'+1}^{r'} \, d_k^* y_k d_k + a z a =
\sum_{k=1}^{m'} \, d_k^* y_k d_k  + 
\sum_{k=m+1}^{r} \, c_k^* x_k c_k + aza.$$
It follows that
$$\sum_{k=1}^m \, c_k^*  x_k c_k - \sum_{k=m+1}^{r} \, c_k^* x_k c_k = \sum_{k=1}^{m'} \, d_k^* y_k d_k -
\sum_{k=m'+1}^{r'} \, d_k^* y_k d_k.$$
Similarly for the `imaginary part', and now it is clear that $\hat{u}$ is one-to-one on ncSpan$(K)$ (resp.\ ncSpan$_n(K_n)$).
\end{proof}

{\em Remark.} If $K_c$ is the complexification from \cite{BMcI} of a nc convex set $K$,
then  ${\rm ncSpan}(K_c)$ is equal to ${\rm ncSpan}_{\bC}(K)$. 
Indeed $a^* (x + iy) b = a^* x + ia^* yb$.  
This can be translated into a statement about the nc span in $E$,  if  we allow rectangular matrices in the expressions
in the span.
E.g.\ if $g \in M_n(E)$  suppose that $g$ is a finite complex nc combination in $E_c$ 
of terms of form $a^* (x+iy) a$.  Apply the map $c$ to both sides.   Compressing this by $\vec e_1 = [I \; \; 0]$ we obtain  $g = \vec e_1^T c(a)^*
c(x,y) c(a) \vec e_1 + \cdots$,  a `cross level combination'.

 \subsection{Complex nc regularity}  \label{cncc}  In our study of nc regular embeddings we insist that 
 our LCTVS  $E$ is a dual operator space,  for reasons discussed in the introduction. 
 For $\bF = \bC$ the nc case is somewhat similar to the basic complex case (Section \ref{coca}).    
 In particular, Theorem  \ref{chco} has a nc variant Theorem  \ref{chco3}. 
One difference though is that `spans' at higher levels are spans with `matrix coefficients', as in Section \ref{ncs}. 
 
  Let $K$ be a nc compact convex set in  a dual operator space $E = F^*$.   Since sometimes we will be embedding 
  a nc compact convex set in one LCTVS, into a different LCTVS such as $\bA(K)^*$ for example.
  We stress that the embedding of $K$ in $E$ needs to be a 
 nc affine map (which is one-to-one and a homeomorphism onto its compact range at each level). 
  Suppose further that $K$ is nc spanning in either of the two senses in  Section \ref{ncs}.   We have a canonical  linear map $\theta : F \to \bA(K)$ given by `restriction' $\varphi_{|K}$ of $\varphi \in F$ to $K$ (on $K_n$ this is $(\varphi^{(n)})_{|K_n}$).  
 We say in this case that $\varphi$ is a {\em linear-extension} of the nc function $(\varphi^{(n)})_{|K_n}$ in $\bA(K)$.
 Similarly we may talk of  {\em linear-extensions} of certain elements $f$ in $M_m(\bA(K)) \cong \bA(K,M_m)$.  These will
 be matrices $[\psi_{ij}] \in M_m(F)$, or may be regarded, via the relation  in our introduction that w*CB$(E,M_n) \cong M_n(E_*)$, 
 as weak* continuous completely bounded linear maps
 $\Psi : E \to M_m$, with $f = (\Psi^{(n)})_{|K_n})$. 
Below all operator spaces are complex and complete.

  \begin{prop} \label{tode}     Let $K$ be a nc compact convex set in  a complex dual operator space $E = F^*$, 
 and suppose that $K$ is nc spanning in either of the two senses in  Section {\rm \ref{ncs}}.  
 Then 
 \begin{enumerate}
        \item [{\rm (1)}]        $\theta$ is completely bounded and one-to-one.
          \item [{\rm (2)}]    $\theta$ is a bicontinuous  isomorphism if and only if 
 $\theta(F) = \bA(K)$ at level 1. That is, if and only if every $f \in \bA(K)$ is  a 
 restriction of some element of $F = M_1(F)$. 
         \item [{\rm (3)}]  $\theta$  is a complete isomorphism if and only if 
 every $f \in \bA(K,M_N)$ has a weak* continuous completely bounded linear-extension 
 $\tilde{f} : E \to M_N$, that is of some element of $M_N(F)$.   Here we can take $N = \aleph_0$.
  \end{enumerate}  
     \end{prop}
        \begin{proof}  (1)\ Clearly $\theta$ is one-to-one by the spanning hypothesis.  Indeed in this case  
 if $f \in F$ is zero on $K$ (resp.\ $K_n$) then it is zero on $E$ (resp.\ $M_n(E)$).   
We further 
 have that $K$ is a bounded nc convex set. (Nc boundedness follows automatic from nc compactness by \cite[Proposition 2.5.3]{DK} or the idea therein, which relies on $K_n$ being bounded for some large enough cardinal $n$.  Of course 
 boundedness is automatic in the Banach space variant, where $E$ is a dual Banach space, as a consequence of the principle of uniform boundedness.) 
 It follows that $\| \theta_n(f) \| \leq c 
 \| f \|_n$ for $f \in M_n(F)$, where $c$ is a positive bound on 
 $K_n$ in $M_n(E)$ for all $n$.  
 So 
 $\theta$ is completely bounded and one-to-one.   
 
  (2)\ By (1) and the open mapping theorem, $\theta$ is a continuous isomorphism if and only if 
 $\theta(F) = A(K)$ at level 1. 
  
 (3)\ Similarly, $\theta$ is a complete isomorphism if and only if 
 $\theta(F) = \bA(K)$ at all levels, or even just at level $N = \aleph_0$.  That is, 
 every $f \in \bA(K,M_N)$ has a linear-extension in $M_N(F)$.  For the map
 $\theta_N : M_N(F) \to M_N(\bA(K))$ is bounded and one-to-one, thus it is a bicontinuous isomorphism (so that $\theta$ is a complete isomorphism) if it is surjective.   That is, if  every $f \in \bA(K,M_N)$ has  a 
weak* continuous completely bounded  linear-extension 
 $E \to M_N$, that is of some element of $M_N(F)$ (see the fact in our introduction that w*CB$(E,M_n) \cong M_n(E_*)$). 
 The converse is similar but easier.  \end{proof}
 
 We will see in Theorem \ref{chco3} that it is enough to take $N = 1$ in the last statement of the Proposition
 if the embedding of $K$ in $E$ is nc preregular (this is defined above 
        Proposition \ref{dum}).

  We say that (nc preregularly embedded) $K$ is {\em nc regularly embedded} in $E$ if 
 $\theta^*$ is a complete isomorphism (that is, surjective and completely bicontinuous).  This is equivalent to $\theta$ 
 being a complete isomorphism, which we characterized above as a uniform affine extension property.  
 
 \medskip
 
  {\bf Remark.} Actually we shall see
 that the `stately' condition in the definition of nc preregularity is not needed for 
 nc regularity. 
 
\begin{cor}   \label{wascorr} If 
 the embedding of a nc convex set $K$ in  a complex dual operator space $E$ is
  nc regular, then the embedding of $K_1$ in $E$ is regular.
 \end{cor}  
 
   \begin{proof}  We already showed that the embedding of $K_1$ in $E$ is preregular.
If $f \in A(K_1)$ then   $f$   is level 1
of a nc function $g$ in $\bA(K)$ by \cite[Theorem 2.5.8]{DK}.
Now $g$ has a continuous completely bounded  linear extension to $E$ by Proposition  \ref{tode},
hence $f$ has a continuous linear extension to $E$.
Thus  the embedding of $K_1$ in $E$ is regular.  
 \end{proof} 

 {\bf Remark.} Suppose that 
  $\theta(x) = 1_{\bA(K)}$ for $x \in F$.
   Then $x$  can be viewed as a nc affine map 
 $E \to \cM$ whose value at $k \in K_n$ is $I_n$.
 We obtain a nc hyperplane $H$ 
 with $H_n$ the matrices in $M_n(E)$ which this map $x$ maps to $I_n$.  So $K \subset H$.
 Also the value of this map at $0_n$ is 0, so that $0_n  \notin H_n$.

        \begin{prop} \label{duma}  If $K = {\rm ncS}(V)$ for a complex operator system $V$ then the  embedding $K \subset V^*$
        is nc regular. 
        \end{prop}
        
        \begin{proof}
       We already remarked in the last section 
 that  the embedding $K \subset E = V^*$ is stately.
 If $\varphi \in M_n( \bA(K)^*) = CB(\bA(K),M_n)$ then 
 $\varphi$ is a linear combination of four completely positive maps.
 Hence it is a nc linear combination of four nc states, 
 hence is a nc linear combination of four maps of form $\delta_x$ for $x \in K_n$. 
 Indeed any cp map in $CB(\bA(K),M_n)$ equals $c \delta_x(\cdot) c$
 for positive $c \in M_n$.   Thus $\bA(K)^*$ is the nc span of ${\rm ncS}(\bA(K))$ in the sense of Section \ref{ncs}.
  Indeed $M_n(E)$ is a nc span of $K_n$.  
  Every $f \in \bA(K)$ has a (weak*) continuous linear-extension to $E$, 
 which is an element of $F = V$.
 Indeed $V = \bA(K)$ so this extension is the obvious one.   In fact, every $f \in M_n(\bA(K))$ has a (weak*) continuous linear-extension
 mapping $E$ to $M_n$.  Indeed 
 $M_n(\bA(K)) \cong \bA(K,M_n) \cong M_n(V)$, which may be viewed as weak* continuous maps $E \to M_n$.
 \end{proof}
 
  Define $r : \bA(K)^* \to E$
 to be $\theta^*$, which is also completely bounded.  Note that 
 $$r_n([g_{ij}])([\psi_{kl}]) = [g_{ij}(\theta(\psi_{kl})] =  [g_{ij}((\psi_{kl})_{| K})].$$ 
 Note that $r_n(\delta_x) = x$ for $x \in K_n$.
 Indeed this follows from the relation 
 $$\langle \varphi , r_n(\delta_x)  \rangle = \langle
  \theta ( \varphi ) , \delta_x \rangle = \delta_x (  \varphi_{|K} ) = \langle \varphi , x \rangle,  \qquad x \in K, \varphi \in F.$$
Here $\langle \varphi , x \rangle = [\varphi (x_{ij})]$ for example, where $x = [x_{ij}] \in K_m$ say.   It follows that $\theta^*$ is surjective
(since $K$ is nc spanning). 
We see next that $\theta^* = q^{-1}$ if $K$ is nc preregularly embedded in $K$.

\begin{prop} \label{filler}  Let $K$ be a (complex) nc compact convex set in  a complex dual operator space $E = F^*$.
 If  $K$ is  nc preregularly embedded in $E$ 
 then the map $\delta : K \to \bA(K)^*$ extends (necessarily uniquely)  to a 
linear isomorphism $q : E \to \bA(K)^*$.     Indeed $q_n : M_n(E) \to CB(\bA(K),M_n)$ 
is a  linear isomorphism for all $n$.  We also have $q^{-1} = r =  \theta^*$ for $\theta$ as above, and this is 
completely bounded.  Finally, $\theta$ has  dense range  in $\bA(K)$. \end{prop}
      
        \begin{proof} The uniqueness follows since  $E = {\rm ncSpan}(K)$. 
         By Section \ref{ncs}  if  $K$ is  nc preregularly embedded in $E$
                  then the map $\delta : K \to \bA(K)^*$ extends to a one-to-one linear map $q : E \to \bA(K)^*$.  Indeed 
 $$q_n (c_1^* x_1 c_1 - c_2^* x_2 c_2 + i
(c_3^* \,x_3 c_3 - c_4^*  \, x_4 c_4)) = c_1^* \, \delta_{x_1} c_1 - c_2^*  \, \delta_{x_2} c_2 + i(c_3^* \, \delta_{x_3} c_3 - c_4^*  \, \delta_{x_4} c_4)
$$ is well defined from  $M_n(E)$ into $CB(\bA(K),M_n)$.   It is also linear, and surjective onto $CB(\bA(K),M_n)$.   

Claim: $q^{-1} = r = \theta^*$.   At level 1 this is easy. 
Clearly in the notation above, $$r_n(q_n(c_1^* x_1 c_1 - c_2^* x_2 c_2 + i
(c_3^* \,x_3 c_3 - c_4^*  \, x_4 c_4)))
= c_1^* x_1 c_1 - c_2^* x_2 c_2 + i
(c_3^* \,x_3 c_3 - c_4^*  \, x_4 c_4).$$
In particular, $r_n = (q^{-1})_n$ for each $n$.
This implies that  $\theta(F)$ is dense in $\bA(K)$, since the annihilator of  $\theta(F)$ is Ker$(\theta^*) = {\rm Ker} (q^{-1}) = (0)$. 
Alternatively, $$\langle \varphi , \theta^*(q(x)) \rangle = \langle
  \theta ( \varphi ) , q(x) \rangle = q(x) (  \varphi_{|K} ) = \langle \varphi , x \rangle,  \qquad x \in E, \varphi \in F.$$ 
   Indeed this is true for $x \in K$, so true on $E = {\rm ncSpan}(K)$.   So $\theta^*(q(x))  = x$.   \end{proof}

Continuing the last proof, we have $q(x)(f) = \langle x , \theta^{-1}(f) \rangle$ for $x \in E, f \in \bA(K)$. Thus  $x \mapsto q(x)(f)$ defines a linear functional $\tilde{f}$  on $E$ extending $f$.  Moreover it is the unique linear-extension of $f$ to $E$ since $E  = {\rm ncSpan} ( K)$.   Similarly for 
$f \in M_n(\bA(K))$, that is for a continuous nc affine $f : K \to M_n$, 
$x \mapsto [q(x)(f_{ij})]$ defines a linear map $E \to M_n$ extending $f$, which is the unique extension.  We now consider continuity of these extensions.

 If the embedding of $K$ in  $E$ is nc preregular 
and if the map $q : E \to \bA(K)^*$ is continuous with 
 respect to the weak* topology of $\bA(K)^*$ then for all $f \in \bA(K)$ the map 
  $\tilde{f}$ in Proposition \ref{tode}  is a weak* continuous extension  (in the sense mentioned above Proposition \ref{tode}) 
   of $f$ to $E$.  Indeed as in 
  the early sections above $\tilde{f}$ is continuous on $E$ for all $f \in \bA(K)$ if and only if $q$ is continuous with 
 respect to the weak* topology of $\bA(K)^*$.

  \begin{thm} \label{chco3}  Let $K$ be a (complex) compact nc convex set  in   
  a complex dual operator space $E$, such that $E = {\rm ncSpan}(K)$.  
  The following are equivalent: 
    \begin{enumerate}
        \item [{\rm (1)}] The embedding of $K$ in  $E$ is nc regular.
        \item [{\rm (2)}]    For some cardinal $N \geq \aleph_0$, and for every $n \leq N$, every $f \in M_n(\bA(K))$, that is, every nc affine continuous $f : K \to M_n$, has a  weak* continuous completely bounded linear-extension to $E$.  
         \item [{\rm (3)}]  The embedding of $K$ in  $E$ is nc preregular, and every $f \in \bA(K)$ 
  has a  weak* continuous   linear-extension to $E$.  
    \item [{\rm (4)}]  The `restriction map' $\theta  : E_* \to \bA(K)$ is a complete isomorphism (or equivalently  $(\theta^*)^{-1} : E \to \bA(K)^*$  is a complete isomorphism). 
        \item [{\rm (5)}] The map $\delta : K \to \bA(K)^*$ extends (necessarily uniquely)  to a complete isomorphism and weak* homeomorphism $\rho : E \to \bA(K)^*$.  
  \end{enumerate}   In this case the map $q$ above is $\rho$ and also equals $(\theta^*)^{-1}$.  Also $F = E_*$ may be made 
  into an abstract operator system (with possibly an equivalent norm) 
  with $M_n(F_+)$  defined to be the $M_n$-valued linear maps on $E$ which are positive (as a nc function) on $K$, and with `identity' 
  the unique such map on $E$ that is 1 on $K$.  This operator system is unitally  completely order isomorphic to 
 $\bA(K)$  via $\theta$. 
   \end{thm} 

 \begin{proof}  (1) $\Rightarrow$ (2) \ If the embedding of $K$ in  $E$ is nc regular then we may assume that $E = \bA(K)^*$,
 and we explained in the proof of Proposition  \ref{duma} that every nc affine continuous $f : K \to M_n$, has a  weak* continuous extension to $E$, for all $n$.  
 
 (2) $\Leftrightarrow$  (4) \ Proposition  \ref{tode}.
 
 (2) $\Rightarrow$  (1) \ By the last line the restriction map $\theta  : E_* \to \bA(K)$ is a complete isomorphism.  Hence $(\theta^*)^{-1} : E \to \bA(K)^*$  is a complete isomorphism and weak* homeomorphism.  
 Note that $$(\theta^*)^{-1}(k)(f) = \langle k , \theta^{-1}(f) \rangle = f(k) = \delta_k(f), \qquad f \in \bA(K), k \in K .$$ 
Thus the embedding of $K$ in $E$ is nc regular, since it is taken by $(\theta^*)^{-1}$ to the nc regular
embedding of $K$ in $\bA(K)^*$.   By the uniqueness in Proposition  \ref{filler}  we have $q = (\theta^*)^{-1}$.

It is clear that (1) and (2) imply (3), and these imply (5).

(5) $\Leftrightarrow$  (4) \ If (5) holds then the predual of the inverse of the weak* homeomorphism $E \to \bA(K)^*$
is a complete isomorphism $\rho : F \to A(K)$.   We have 
$$\rho(\psi)(k) = \delta_k(\rho(\psi)) = \langle k , \psi \rangle, \qquad \psi \in F, k \in K .$$
That is $\rho = \theta$.   

(3) $\Rightarrow$  (2) \ If (3) holds then, as we said before the theorem, 
for all $f \in \bA(K)$ the map 
  $\tilde{f}$ above  is a weak* continuous 
  extension of $f$ to $E$.  Hence as we said in the proof of Corollary 
  \ref{wascorr}, any $f \in A(K_1)$
  has a  weak* continuous extension to $E$. 
 By Corollary   \ref{wascoro} the embedding of $K_1$ in $E$ is preregular, and indeed is regular and 
we may assume that $E$ is a dual space with the weak* topology by
Theorem \ref{chco2}.   Hence  for all continuous nc affine $f : K \to M_N$, the map $x \mapsto \tilde{f}(x) = [q(x)(f_{ij})] = [\widetilde{f_{ij}}(x)]$  is a weak* continuous extension of $f$ to $E$.   This follows for example from \cite[Corollary 1.6.3 (1)]{BLM} and the Krein-Smulian 
 theorem: 
  for a bounded net  $x_t \to x$ in $E$ then $\widetilde{f_{ij}}(x_t) \to \widetilde{f_{ij}}(x)$ for all $i, j$, hence 
 $[\widetilde{f_{ij}}(x_t)] \to [\widetilde{f_{ij}}(x)]$ weak* in $M_N$, by that Corollary. 
 
The last assertions are easy, similar to the ones in Theorems \ref{chco} and \ref{chco2}.
 \end{proof}

{\bf Remarks.}   1)\ Again, that the canonical embedding of 
${\rm ncS}(V)$ in $V^*$ is nc preregular is easy in the nc complex case.
However that the canonical embedding of 
${\rm ncS}_{\bC}(V)$ in $V^*$  is a nc regular embedding is not obvious.  Indeed it essentially is precisely the complex case of Davidson and Kennedy's Kadison-Webster-Winkler duality formula $V \cong \bA({\rm ncS}(V))$ unitally complete order isomorphically, since this implies that  every $f \in M_N(\bA({\rm ncS}(V)))$ has a 
weak* continuous extension to $V^*$, and such extensions come from matrices in $M_N(V)$.
Recall that the canonical map $\iota : V \to \bA(S(V))$
 is essentially our map $\theta$ above when $E = V^*$ and $F = V$.   
 
 \smallskip
 
 2)\ Because in the last theorem $F = E_*$ may be made 
  into an abstract operator system, $E$ is an abstract  nc dual base normed space in the sense of \cite{BH}, using also  Theorem 5.5 in that paper.  
 
 \medskip
 
 The following is a `completely isometric variant'.   It is rather superficial but we include it for completeness. 
 
  \begin{cor} \label{chco24} For a complex dual operator space $E$ and a nc compact convex set $K$ such that $E = {\rm ncSpan}(K)$, with $K$ contained as a nc set in $({\rm Ball}(M_n(E)))$, the following are equivalent:
      \begin{enumerate}
        \item [{\rm (1)}] The embedding of $K$ in  $E$ is nc regular, 
         and the map $q$ above is a complete contraction. 
  \item [{\rm (2)}]  For some cardinal $N \geq \aleph_0$, and for every $n \leq N$, every  $f \in M_n(\bA(K))$, that is, every nc affine continuous $f : K \to M_n$, has a  weak* continuous linear-extension to $E$ whose 
 completely bounded norm  is $\| f \|$. 
   \item [{\rm (3)}]   The map $\theta^*$ is a complete isometry (hence is a surjective completely  isometric isomorphism
    $E \cong \bA(K)^*$). 
   \item [{\rm (4)}]   The map $\theta$ above is  a surjective completely  isometric isomorphism between $F = E_*$ and $\bA(K)$.  
  \item [{\rm (5)}]   The map $\delta : K \to \bA(K)^*$ extends (necessarily uniquely)  to a 
  weak* homeomorphic linear complete isometry  $\rho : E \to \bA(K)^*$  (both spaces with the weak* topology).
    \end{enumerate} 
  In this case 
  $\theta^* = q^{-1} = \rho^{-1}$.   \end{cor} 

\begin{proof}  By the previous result any of (2)--(5) imply that the embedding of $K$ in  $E$ is nc regular, 
and $\theta^* = q^{-1} = \rho^{-1}$.
Then (5) is saying that $q = (\theta^*)^{-1}$ is a completely   isometric isomorphism, which implies (2).  
Clearly (4) is equivalent to (3), 
and these  imply (1), which implies (5). 
 Note that any nc affine continuous contractive (w.r.t.\ the sup norm) $f : K \to M_n$ extends to a contractive matrix in $M_n(V)
= CB((M_n)_*,V) = CB^\sigma(V^*,M_n)$.   Thus $f$ comes from a completely contractive (weak*) continuous $\tilde{f} :
E \to M_n$.
 It is then easy to see as in the last proof that (2)  
 is equivalent to (4), and that $\theta^* = q^{-1}$. 

 (1) $\Rightarrow$  (3) \  If $\Psi = [\psi_{ij}] \in {\rm Ball}(M_n(E_*))$ then we may regard 
 $\Psi$ as a complete contraction $E \to M_n$.   If $x \in K_m \subset {\rm Ball}(M_m(E))$
 it follows that $\| \Psi^{(m)}(x) \| \leq 1$.  Thus $\theta$  is a complete contraction on
 $F = E_*$.   Hence $\theta^* = q^{-1}$  is a complete contraction.  Thus (3) holds.  
  \end{proof}

\section{Real nc regularity} \label{rncr}  

In this section we use the real nc convexity notation and theory from \cite{BMcI}.  Thus $E_c$ is the Ruan complexification of the operator space $E$, and $K_c$ is the complexification from \cite{BMcI} of a nc convex set $K$.  We defined $c(x,y)$ in the introduction. 

 Sadly however the real case of nc regularity looks a bit more complicated in a certain regard.   
The issue is that it seems one cannot expect ncSpan$(K)$ to be helpful in the way it was in the complex case.  A slightly different approach is needed.

Let  $K$ be a real nc compact convex set embedded in a real dual operator space $E$, with $K$ lying  in a  closed nc hyperplane $H$ not passing through 0.
That is, there exists $x \in F = E_*$ such that $H_n$ consists of the matrices in $M_n(E)$ which $x$ maps to $\gamma \, I_n$, fixed $\gamma \in \bR, \gamma \neq 0$.  Then $E_c$ has a real subspace $W = E$ with $W + iW = E_c$ and $K \subset W$ as graded nc sets.  
Then $K_c$ lies in the nc hyperplane in $E_c$ defined by $x$.  Indeed if $x+iy \in K_c$ then $c(x,y) \in K$, and 
$f$ evaluated at $c(x,y)$ is $I \oplus I$.   Hence 
$f$ evaluated at $x$ and $y$ is $I$ and $0$ respectively, so that $f$ evaluated at $x+iy$ is $I$. 

We will say that the embedding of $K$ in $E$ is nc {\em preregular} if 
 $K$ lies  in a nc hyperplane as above, and $E_c = {\rm ncSpan}(K_c)$ at each level.     
 
For a cardinal $n$ we define as in  \cite{BMcI} the isometry $u_n = \frac{1}{\sqrt{2}} \begin{bmatrix} 
    1_n \\ -i \cdot 1_n 
\end{bmatrix}$ where $1_n$ is the $n$-dimensional identity operator.  We sometimes also write this as $u$. 
 Again we write  $\theta : F = E_* \to \bA(K)$ for the canonical `restriction map', defined analogously to 
the first paragraphs of Subsection \ref{cncc}.  We also define `linear-extensions' to $E$ as we did there.

  \begin{prop} \label{tore}     Let $K$ be a real nc compact convex set in  a dual real operator space $E = F^*$, 
 and suppose that $K_c$ is nc spanning in $E_c$ in either of the two senses in  Section {\rm \ref{ncs}}.  
 Then 
 \begin{enumerate}
        \item [{\rm (1)}]        $\theta$ is completely bounded and one-to-one. 
          \item [{\rm (2)}]    $\theta$ is a bicontinuous  isomorphism if and only if 
 $\theta(F) = A(K)$ at level 1. That is, if and only if every $f \in A(K)$ is  a 
 restriction of some element of $F = M_1(F)$. 
         \item [{\rm (3)}]  $\theta$  is a complete isomorphism if and only if 
 every $f \in \bA(K,M_N)$  has  a 
  weak* continuous completely bounded linear-extension 
 $E \to M_N$, or equivalently some element of $M_N(F)$.   Here $N = \aleph_0$.
  \end{enumerate}  
     \end{prop}
        \begin{proof}  (1)   As in the proof of Proposition \ref{tode} nc boundedness of $K$ is automatic in the real case. It follows that $\| \theta_n(f) \| \leq \frac{1}{c} \| f \|_n$ for $f \in M_n(F)$, where $c$ is a positive bound on 
 $K_n$ in $M_n(E)$ for all $n$.  
 So 
 $\theta$ is completely bounded.  Suppose that $\theta(\psi) = 0$ for $\psi \in F$.   For $x + iy \in (K_c)_n$ we have $c(x,y) \in F_{2n}$
 so that $0 = \theta(\psi) (c(x,y))$.   Note that $(\psi_c)_{|K_c}(x+iy)$ equals $$\theta_c(\psi + i0)(x+iy) = \psi(x) + i \psi(y) 
 = (\theta(\psi))_c(x + iy) = u^* \theta(\psi)(c(x,y)) u = 0.$$   
 Here $u = \frac{1}{\sqrt{2}} [ 1 \; -i]^\tran$ as in e.g.\ \cite{BMcI}. 
 Since this is true for all $x + iy \in K_c$ and $K_c$ is nc spanning we have $\psi_c = 0$ and $\psi = 0$.
So  $\theta$ is  one-to-one.   
 
  (2) and (3) are just as in Proposition \ref{tode}.  \end{proof}

  We say that (nc preregularly embedded) $K$ is {\em nc regularly embedded} in $E$ if 
 $r = \theta^*$ is a complete isomorphism (that is, surjective and completely bicontinuous).  This is equivalent to $\theta$ 
 being a complete isomorphism, which we characterized above as a uniform  affine extension property.

  \bigskip
 
 {\bf Remarks.} 1) \ Suppose that 
  $\theta(x) = 1_{\bA(K)}$ for $x \in F$. 
  Then $x$ s a nc affine map 
 $E \to \cM$ whose value at $k \in K_n$ is $I_n$.
 We obtain a nc hyperplane $H$ 
 with $H_n$ the matrices in $M_n(E)$ which this map $x$ maps to $I_n$.  So $K \subset H$.
 Also the value of this map at $0_n$ is 0, so that $0_n  \notin H_n$.
 
 \medskip
 
 2) \ If  the embedding of a nc convex set $K$ in  a real dual operator space $E$ is
  nc preregular (resp.\ nc regular), one cannot deduce as we did in 
  Corollary \ref{wascoro}, that the embedding of $K_1$ in $E$ is preregular (resp.\ regular).   Indeed
  $K_1$ may be $(0)$, as it is in the case of the quaternions for example.   One would expect,
  in the light of a result from \cite{BMcI}, that the embedding of $K_2$ in $M_2(E)$ is appropriately preregular
  (resp.\ regular).
  Indeed 
  this is the case. Certainly $K_2$ spans  $M_2(E)$ by a variant of the $A,B,C,D$  trick in the 
  proof of Corollary   \ref{wascoro}.  Suppose that $\psi \in F = E_*$ defines a nc hyperplane $H$ containing $K$, with 
  $H_n$ consisting of the matrices in $M_n(E)$ which $\psi$ maps to $\gamma \, I_n$, for a fixed $\gamma \in \bR, \gamma \neq 0$.  
  Then $K_2$ lies in the   hyperplane in $M_2(E)$ defined by where the functional $[x_{ij} ] \mapsto \psi(x_{11})$ 
  equals $\gamma$. So the embedding of $K_2$ in $M_2(E)$ is preregular. If  the embedding of $K$ is nc regular, suppose that 
   $f \in A_{\bR}(K_2)$. Then $f \cdot I_2$ is level 2 of a nc function $g \in \bA(K)$ by 
\cite[Theorem 6.12]{BMcI}.   If $\tilde{g} \in F = E_*$ is
a linear-extension of $g$ then $\tilde{g}_2$ extends $f \cdot I_2$ to $M_2(E)$.  Then the 1-1 coordinate of $\tilde{g}_2$ extends $f$ to $M_2(E)$.  Thus  the embedding of $K_2$ in $M_2(E)$ is regular.
 
        \begin{prop} \label{dumar}  If $K = {\rm ncS}(V)$ for a real operator system $V$ then the  embedding $K \subset V^*$
        is real nc regular, and  the  embedding $K_c \subset V_c^*$
        is complex nc regular. 
        \end{prop}
        \begin{proof}
       We already  established the hyperplane condition  in the last section 
 for the embeddings $K \subset E = V^*$ and $K_c \subset V_c^*$.  Clearly $K_c \subset V_c^*$ is stately with $W$ the copy of $V^*$ in $V_c^*$.   Indeed since $\bA(K_c) = \bA(K)_c$ \cite{BMcI} we have from Proposition  \ref{duma} that 
 the  embedding $K_c \subset V_c^*$     is complex nc regular.    In particular  $E_c = {\rm ncSpan}(K_c)$.
  The rest is as in  Proposition  \ref{duma}. 
 \end{proof}

  \begin{thm} \label{chco3r}  Let $K$ be a real nc  compact convex set  in   
  a real dual operator space $E$, such that $E_c = {\rm ncSpan}(K_c)$.  
  The following are equivalent:
    \begin{enumerate}
        \item [{\rm (1)}] The embedding of $K$ in  $E$ is nc regular.
        \item [{\rm (2)}]    For some cardinal $N \geq \aleph_0$, every $f \in M_n(\bA(K))$, that is, every nc affine continuous $f : K \to M_n$, has a  weak* continuous linear-extension to $E$ if $n \leq N$.  
         \item [{\rm (3)}]  The embedding of $K$ in  $E$ is real nc preregular, and every $f \in \bA(K)$ 
  has a  weak* continuous linear-extension to $E$.  
     \item [{\rm (4)}]  The restriction map $\theta  : E_* \to \bA(K)$ is a complete isomorphism (or equivalently  $(\theta^*)^{-1} : E \to \bA(K)^*$  is a complete isomorphism and weak* homeomorphism). 
     \item [{\rm (5)}] The map $\delta : K \to \bA(K)^*$ extends (necessarily uniquely)  to a 
complete isomorphism and weak* homeomorphism $\rho : E \to \bA(K)^*$.  
       \item [{\rm (6)}]  The canonical embedding of $K_c$ in  the operator space 
       complexification $E_c$ is (complex) nc regular.
  \end{enumerate}   
In this case the map $q$ above is $\rho$, and also equals $(\theta^*)^{-1}$.  Also $F = E_*$ may be made 
  into an abstract real operator system (with possibly an equivalent norm) 
  with $M_n(F_+)$  defined to be the $M_n$-valued linear maps on $E$ which are positive (as a nc function) on $K$, and with `identity'  the unique such map on $E$ that is 1 on $K$.  This operator system is unitally completely order isomorphic to $\bA(K)$  via $\theta$. 
   \end{thm} 

 \begin{proof}  (1) $\Rightarrow$ (2) \ If the embedding of $K$ in  $E$ is nc regular then we may assume that $E = \bA(K)^*$, so this follows from Proposition  \ref{dumar} .  
 
 (2) $\Leftrightarrow$  (4) \ Proposition  \ref{tore}.
 
 (2) $\Rightarrow$  (1) \ This follows from 
  Proposition  \ref{tore}  as in the matching part of the proof in Theorem \ref{chco3}.   Similarly for $q = (\theta^*)^{-1}$. 
 
It is clear that (1) and (2) imply (3).  The equivalence with (5) is as in Theorem \ref{chco3}. 

(3) $\Rightarrow$  (2) \ If (3) holds then the embedding of $K_c$ in  $E_c$ is stately, with $W = E$ in the definition of stately. 
It is thus complex nc preregular by considerations at the start of this section.
Also, every $f \in \bA(K_c)$ 
  has a  weak* continuous extension to $E_c$.  Indeed by \cite[Theorem 3.9]{BMcI}, we may write $f = f_1 + i f_2$ with $f_k \in \bA(K)$.
By hypothesis each $f_k$ has a  weak* continuous linear-extension $\varphi_k$ to $E$.   Then $\varphi_1 + i \varphi_2$ 
(or $(\varphi_1)_c + i (\varphi_2)_c$)  is 
a weak* continuous linear-extension of $f$ to $E_c$.   Thus by Theorem   \ref{chco3} we have that $K_c$ is  complex nc regular in $E_c$.
Thus the restriction map $\theta_{\bC} : F_c \to \bA(K_c)$ is a completely bounded isomorphism.  Hence so is
 the restriction map $\theta: F \to \bA(K)$ since $\theta_{\bC}(\varphi + i0) = \theta(\varphi)_c$ for $\varphi \in F$.  One way to see 
 the latter 
 is that 
 $$\theta_{\bC}(\varphi + i0)( x+iy) = \varphi(x) + i \varphi(y) = (\theta(\varphi))_c(x+iy), \qquad x + iy \in K_c .$$ 
 Indeed as in the proof of Proposition \ref{tore} we have  $$\theta(\varphi)_c(x+iy) = u^* \theta(\varphi)(c(x,y)) u = u^* c(\varphi(x), \varphi(y))  u = \varphi(x) + i \varphi(y)$$
as needed. 

(2) $\Rightarrow$  (6) \ Suppose that  $f : K_c \to M_n(\bC)$ is a nc affine continuous function.
That is, $f \in M_n(\bA_{\bC}(K_c)) \cong M_n(\bA_{\bR}(K)_c) = M_n(\bA_{\bR}(K))_c$.
Write $f = g_c + i h_c$ for $g, h \in M_n(\bA_{\bR}(K))$, that is for $g, h : K \to M_n(\bR)$  nc affine continuous.
Each has a weak* continuous completely bounded  linear-extension $\tilde{g}, \tilde{h} : E  \to M_n(\bR)$.
Then  $\tilde{g}_c + i \tilde{h}_c = (\tilde{g}+ i \tilde{h})_c$ is a weak* continuous completely bounded  linear-extension of $f$.
Indeed if $x + iy \in (K_c)_m$ then 
$$((\tilde{g} + i \tilde{h})_c)_m(x+iy) = u^* (\tilde{g} + i \tilde{h})(c(x,y)) u 
= u^* (g_{2m}(c(x,y)) + i h_{2m}(c(x,y))) u$$
which equals $(g_c + i h_c)_m(x+iy) =   f_m(x+iy).$ Thus (6) follows from  Theorem \ref{chco3}.

(6) $\Rightarrow$ (2) \ Suppose that  $f : K \to M_n(\bR)$ is a real nc affine continuous function.
Then $f_c : K_c \to M_n(\bC)$ is a complex nc affine continuous function by \cite[Lemma 3.2]{BMcI}.
Hence there is a weak* continuous completely bounded  linear-extension $\widetilde{f_c} : E_c \to M_n(\bC)$.
Then $\pi_n \circ \widetilde{f_c} \circ \iota_E : E \to M_n(\bR)$ is a real weak* continuous completely bounded  map.
Also it is a linear-extension of $f$ since  $(\pi_n ( \widetilde{f_c} ( \iota_E(k))) = f(k)$ for $k \in K_n$.

The last assertions are easy, similar to the ones in Theorems \ref{chco} and \ref{chco2}.  
\end{proof}  

{\bf Remark.} If $K$ is real nc preregularly embedded then we  show that $\theta$ has  dense range  in $\bA(K)$.  Indeed by the start of the proof that 
(3) $\Rightarrow$  (2) we have that the embedding of $K_c$ in  $E_c$ is nc preregular.   Thus 
the restriction map $\theta_{\bC} : F_c \to \bA(K_c)$ has dense range.   
Thus $\theta_{\bC}^*$ is one-to-one.  However by the proof above
$\theta_{\bC} = \theta_c$, and so $\theta_{\bC}^* = (\theta_c)^* = (\theta^*)_c$.   
The last equality is a general fact about the dual of the complexification  of any Banach space: if $u : X \to Y$ then 
$(u_c)^*(f + ig)(x+iy)$ equals $$(f + ig)(u(x)+iu(y)) = (u^*(f) + i u^*(g))(x+iy) = (u^*)_c(f+ig)(x+iy),$$ for $x, y \in X, f, g \in Y^*.$
 So $\theta^*$ is one-to-one, and  so 
$\theta$ has  dense range  in $\bA(K)$.

\begin{cor} \label{filler2}  Let $K$ be a nc compact convex set in  a dual operator space $E = F^*$.
 If  $K$ is  nc regularly embedded in $E$, then the map $\delta : K \to \bA(K)^*$ extends  to a linear complete isomorphism $q : E \to A(K)^*$.     Indeed $q_n : M_n(E) \to CB(\bA(K),M_n)$ 
is a  linear isomorphism for all $n$.  We also have $q^{-1} = \theta^*$ for $\theta$ as above, and this is 
completely bounded.  \end{cor}

     The proof of the following (rather superficial result) is the same as that of Corollary \ref{chco24}:

 \begin{cor} \label{chco25} For a real dual operator space $E$ and a nc compact convex set $K$ such that $E = {\rm ncSpan}(K)$, with $K$ contained as a nc set in $({\rm Ball}(M_n(E)))$, the following are equivalent:
      \begin{enumerate}
        \item [{\rm (1)}] The embedding of $K$ in  $E$ is nc regular, 
         and the map $q$ above is a complete contraction. 
  \item [{\rm (2)}]  For some cardinal $N \geq \aleph_0$, and for every $n \leq N$, every  $f \in M_n(\bA(K))$, that is, every nc affine continuous $f : K \to M_n$, has a  weak* continuous linear-extension to $E$ whose 
 completely bounded norm  is $\| f \|$. 
   \item [{\rm (3)}]   The map  $\theta^*$ is a surjective complete isometry. 
   \item [{\rm (4)}]   The map $\theta$ above is  a surjective completely  isometric isomorphism between $F = E_*$ and $\bA(K)$.  
   \item [{\rm (5)}]   The map $\delta : K \to A(K)^*$ extends (necessarily uniquely)  to a weak* homeomorphic linear complete isometry  $\rho : E \to \bA(K)^*$  (both spaces with the weak* topology).
 \end{enumerate} 
  In this case   $\theta^* = \rho^{-1} = q^{-1}$.   \end{cor}

 {\bf Acknowledgement}: We thank Damon Hay and Arianna Cecco for some questions and comments. 
In a sequel paper  \cite{BP2} we refine some of our regularity results in the present paper using classical and noncommutative base norm spaces.   For example,   For example, if $V$ is an LCTVS  and if $K$ is a compact convex set 
which spans $V$ and lies in a hyperplane in $V$ not passing through 0 (that is, $K$ is preregular in $E$) 
then $V$ is canonically isomorphic to $A(K)^*$ via a continuous  isomorphism  which
is a homeomorphism on bounded sets.
Similarly in the nc case: If $V$ is a complex  $*$-LCTVS, and if $K$ is a selfadjoint compact matrix convex set 
such that $K_1$  spans $V_{\rm sa}$ and lies in a hyperplane not passing through 0 (that is, $K_1$ is preregular in $V_{\rm sa}$) , 
then $\bA(K,\bC)^* \cong V$ via a continuous   selfadjoint  isomorphism  which
is a homeomorphism on bounded sets. 

 \bigskip

  {\bf Funding declaration:} This project was partially supported by National Science Foundation grant DMS-2154903.

\end{document}